\documentclass[11pt]{article}

\usepackage[a4paper, total={6in, 8in}]{geometry}
\usepackage[english]{babel}

\usepackage{graphicx}
\usepackage{wrapfig}
\usepackage{caption}
\usepackage{subcaption}
\usepackage{xcolor}

\usepackage{cite}

\usepackage{amsmath,amsthm,amssymb,mathtools}
\usepackage{esvect}
\usepackage{hyperref}
\usepackage{bbm}

\usepackage{titlesec}
\titleformat*{\paragraph}{\itshape}
\titleformat*{\subsubsection}{\itshape}

\renewcommand{\P}{\mathbb{P}}
\newcommand{\Z}{\mathbb{Z}}

\newcommand{\E}{\mathbb{E}}
\newcommand{\F}{\mathcal{F}}

\DeclarePairedDelimiter{\abs}{\lvert}{\rvert}

\newtheorem{theorem}{Theorem}[section]
\newtheorem{corollary}[theorem]{Corollary}
\newtheorem{lemma}[theorem]{Lemma}
\newtheorem{proposition}[theorem]{Proposition}
\newtheorem{conjecture}[theorem]{Conjecture}
\newtheorem{question}{Question}[section]

\theoremstyle{definition}
\newtheorem{example}{Example}[section]

\theoremstyle{remark}
\newtheorem{remark}[theorem]{Remark}

\theoremstyle{definition}
\newtheorem{definition}{Definition}[section]

\usepackage[bottom]{footmisc}

\title{A study of distributional complexity measures for Boolean functions}%Distributional complexity measures: separations and (polynomial?) relations}
\author{Laurin Köhler-Schindler\footnote{ETH Zürich, Zürich, Switzerland. Email: laurin.koehler.schindler@gmail.com} \and Jeffrey E.\ Steif\footnote{Chalmers University of Technology and Gothenburg University, Gothenburg, Sweden. Email: steif@chalmers.se}}
\date{\today}

\mathtoolsset{showonlyrefs}

\begin{document}

\maketitle

\begin{abstract}
	A number of complexity measures for Boolean functions have previously been introduced. These include
	(1) sensitivity,  (2) block sensitivity, (3) witness complexity, (4) subcube partition complexity and
	(5) algorithmic complexity. Each of these is concerned with ``worst-case'' inputs. It has been shown that there is ``asymptotic separation''  between these complexity measures
	and very recently, due to the work of Huang, it has been established that they are all
	``polynomially related''. In this paper, we study the notion of distributional complexity where the input bits are independent
	and one considers all of the above notions in expectation. We obtain a number of results concerning
	distributional complexity measures, among others addressing the above concepts of ``asymptotic separation'' and being
	``polynomially related'' in this context. We introduce a new distributional complexity measure, local witness complexity, which
	only makes sense in the distributional context and
	we also study a new version of algorithmic complexity which involves
	partial information. Many interesting examples are presented including some related
	to percolation. The latter connects a number of the recent developments in percolation theory over the last two decades
	with the study of complexity measures in theoretical computer science.
\end{abstract}

\tableofcontents

\section{Introduction}\label{sec:introduction}

A Boolean function is an arbitrary function
$$
f:\{0,1\}^n \to \{0,1\}.
$$
Boolean functions arise in a number of areas of mathematics including theoretical computer science and its role in the latter
will be the main topic of this paper.  A number of concepts have been previously introduced which are designed to measure
the ``complexity'' of a Boolean function; see Subsection \ref{subsec:defs-hierarchy-det}
for five of the standard complexity measures which have been intensively studied (see, e.g., \cite{Buhrman2002,Aaronson2021}).  These are
(1) sensitivity,  (2) block sensitivity, (3) witness complexity, (4) subcube partition complexity and (5) algorithmic complexity.
In all of these cases, the definition is given in terms of ``worst case'' over the possible input strings from
$\{0,1\}^n$. 
There is an ordering of these complexity notions which holds for all Boolean functions; see
Proposition \ref{prop:hierarchy-det}. We will often use the term \emph{deterministic complexity measures} in order
to distinguish these from the  \emph{distributional complexity measures} introduced further down.
Two particular directions which have been successfully investigated within this field are
(1) asymptotic separation and (2) polynomial relations. Asymptotic separation means that for any two of the complexity measures (1)--(5),
there exists a sequence of Boolean functions $(f_n)_{n\ge 1}$ so that the ratio of these two
complexity measures applied to $f_n$ approaches $\infty$ (or zero); see Theorem \ref{thm:asymp-sep-det}.
The term ``polynomial relations''  means that all of these complexity measures are related by universal
polynomial factors; more precisely, there is a constant $C$ so that each complexity measure is at most any other
complexity measure raised to the power $C$; see Theorem \ref{thm:poly-rel-det}.
Recently in 2019, Huang obtained the final step of this result by proving the so-called sensitivity conjecture \cite{Huang2019}.

Rather than considering ``worst-case'' input strings, which is done in the above notions, one can measure complexity in an average sense.
One of the main books on Boolean functions when one considers ``average case'' notions is  \cite{ODonnell2014}.
Here one assumes that the input string $\{0,1\}^n$ is chosen randomly; usually one assumes that
the bits are chosen independently, 1 or 0, with probabilities $p$ and $1-p$ for a given $p$. Then one considers the various
notions mentioned earlier ((1) sensitivity,  (2) block sensitivity, (3) witness complexity, (4) subcube partition complexity and
(5) algorithmic complexity)
on expectation rather than on a ``worst-case'' input string. We use the term distributional complexity in this context;
see Subsection \ref{subsec:defs-hierarchy-dist} for the precise definitions.  One of the various goals of this paper is to study
the notions of  ``asymptotic separation'' and ``polynomial relations'' in the context
of distributional complexity; see Theorem \ref{thm:asymp-sep-dist} for the former and
Subsection \ref{subsec:polynomial-relations-dist} for the latter. 
Some other related references are \cite{Blanc2024} and \cite{Sanyal2024}.

In Subsection \ref{subsec:defs-hierarchy-dist}, we also introduce a new complexity measure, called local witness
complexity, which does not have an analogue in the deterministic (worst-case) setting.
This notion of local witness complexity is very natural from a probabilistic perspective, seemingly more natural than the closely aligned notion of subcube partition.
In addition, it often captures the minimal structure needed for certain arguments to work; see Theorem \ref{thm:odonnell-servedio} for an example of this.
An analogous notion of local sets also arises in the study of the so-called Gaussian Free Field.

As in the deterministic case, there is an ordering of these distributional complexity measures, including local witness complexity,
which holds for all Boolean functions; see Proposition \ref{prop:hierarchy-dist}. We also give examples
demonstrating that this new notion of local witness complexity is distinct from all of the other complexity measures.

We proceed to give a brief summary of different parts of the paper. First, many interesting examples of Boolean
functions together with either their deterministic complexity measures or
their distributional complexity measures (or both) are given throughout the text.  A particularly interesting class of
Boolean functions is the class of \emph{percolation functions}; see Subsection  \ref{subsec:percolation-examples} where
these are studied. Within these, the sequence corresponding to crossings of a square is especially interesting
(see Example \ref{ex:square-crossing})
since many deep results in percolation theory lead to fairly detailed information
concerning the various distributional complexity measures.
Second, there is a natural notion of composing Boolean functions, and this naturally gives rise, by iteration, to various sequences of Boolean functions.
How these different complexity measures, both in the deterministic and in the distributional sense, behave under
composition is recalled (for deterministic) and studied (for distributional) in Section \ref{sec:composition}.
In Subsection \ref{subsec:subcube},  we examine
the distinction between algorithmic and subcube partition complexity using a geometric perspective. In particular, we will
see that for subcube partition complexity, one can consider the two parts of the cube where $f$ is 1, respectively 0, separately, but
this cannot be done for algorithms.
In Subsection \ref{subsec:stopping-set-perspective}, we introduce a unifying perspective on both
subcube partitions and local sets. This is done via the concept of a \emph{stopping set} which is analogous to
the concept of a stopping time in probability theory.
In Section \ref{sec:alg-comp-partial-info}, we introduce a new notion
of distributional algorithmic complexity with partial information which appears interesting and has certain similarities with the notion of fractional query algorithms introduced in \cite{Gross2022}. In addition, it is potentially relevant
for studying asymptotic separation of two of the notions of distributional complexity, namely algorithmic and subcube partition; see end of Section \ref{sec:asymp-separation}.

Results from theoretical computer science  have in recent years led to new results (or simplifications of older results)
in percolation theory and statistical mechanics, such as so-called sharp thresholds for various models. We therefore hope that this paper will motivate
the percolation community to look more systematically at distributional complexity measures and at the same time
to motivate the theoretical computer science community to study distributional complexity measures (not only worst-case)
even more systematically given that these are very relevant from a probabilistic perspective.
In addition, a systematic study of distributional complexities for non-product measures would also be an interesting direction to pursue in future research. For example, an important inequality, called the OSSS inequality, has recently been extended in this direction, leading to the resolution of important conjectures in statistical mechanics \cite{DuminilCopin2019}.

\section{Definition and hierarchy of complexity measures}\label{sec:defs-hierarchy}

\subsection{Deterministic complexity measures}\label{subsec:defs-hierarchy-det}

In this subsection, we review deterministic complexity measures for Boolean functions. We refer to the coordinates of the $n$-dimensional hypercube $\{0,1\}^n$ as bits. An \emph{algorithm} $A$ is an order of querying the bits (i.e.\ asking for the values) one by one, where the next query is allowed to depend on the outcome of all previous queries. Note that an algorithm will have queried all bits eventually. For a Boolean function $f: \{0,1\}^n \to \{0,1\}$ and a realization of the bits $x \in \{0,1\}^n$, we write $c_f(A,x)$ for the number of queries made by $A$ on $x$ until $f$ is determined. The \emph{deterministic algorithmic complexity} of $f$ is defined as
\begin{equation} \label{eq:def-det-algo-comp}
	a_D(f) = \min_A \ \max_x\  c_f(A,x).
\end{equation}

A \emph{subcube partition} $\mathcal C$ is a partition of the hypercube $\{0,1\}^n$ into subcubes.   It will be convenient to identify a subcube $C \subseteq \{0,1\}^n$ with a sequence $(c_1,\ldots,c_n) \in \{0,1,\star\}^n$, where $c_i \in \{0,1\}$ means that bit $i$ is fixed to be $0$ respectively $1$, and $c_i = \star$ means that bit $i$ is not fixed. We write $I_C$ for the set of bits fixed by $C$. Given a subcube partition $\mathcal C$, we denote the subcube containing $x$ by $C(x) \in \mathcal C$ and its co-dimension by 
\begin{equation}
	c(\mathcal C,x) = n - \log_2(\abs{C(x)}).
\end{equation}
In other words, $c(\mathcal C,x)$ equals $\abs{I_{C(x)}}$, the number of bits fixed by $C(x)$, and $\log_2(\abs{C(x)})$ is the number of bits not fixed by $C(x)$. We say that a subcube partition $\mathcal C$ determines $f$, denoted by $\mathcal C \sim f$, if it is constant on each subcube $C \in \mathcal C$. The \emph{deterministic subcube partition complexity} of $f$ is defined as
\begin{equation} \label{eq:def-det-sc-comp}
	sc_D(f) = \min_{\mathcal C \sim f}\ \max_x\  c(\mathcal C,x).
\end{equation}

A set $W \subseteq [n]$ is called a \emph{witness set} for $f$ and $x$ if $x\vert_W$ determines $f$. Given $x$, we refer to the minimum of $\abs{W}$ over all witness sets $W$ for $f$ and $x$ as the \emph{minimum witness size} and denote it by $w_f(x)$. The \emph{deterministic witness complexity} of $f$ is defined as
\begin{equation}
	w_D(f) = \max_x\  w_f(x).
\end{equation}

We denote by $b_f(x)$ the maximum number $m$ of disjoint blocks $B_1,\ldots,B_m \subseteq [n]$ such that
\begin{equation} \label{eq:block-pivotality}
	f(x^{B_i}) \neq f(x), \quad \forall\, 1 \le i \le m,
\end{equation}
where $x^{B_i}$ denotes the realization of the bits obtained from $x$ by flipping all bits in $B_i$. The \emph{deterministic block sensitivity} of $f$ is defined as
\begin{equation}
	b_D(f) = \max_x\  b_f(x).
\end{equation}

Finally, we denote by $s_f(x)$ the number of bits $i$ such that $f(x^i) \neq f(x)$, where $x^i := x^{\{i\}}$. These bits are called \emph{pivotal} for $f$ at $x$. The \emph{deterministic sensitivity} of $f$ is defined as
\begin{equation}
	s_D(f) = \max_x\  s_f(x).
\end{equation}

Although well-known, we give the proof of the following proposition for the sake of the reader.
\begin{proposition}\label{prop:hierarchy-det} For every Boolean function $f$, the deterministic complexities satisfy
	\begin{equation}
		a_D(f) \ge sc_D(f) \ge w_D(f) \ge b_D(f) \ge s_D(f).
	\end{equation}
\end{proposition}

\begin{proof}
	We prove the four inequalities one by one. 
	First, consider an algorithm $A$ achieving the minimum in \eqref{eq:def-det-algo-comp}. When stopped at the moment of determining $f$, it induces a subcube partition $\mathcal C(A) \sim f$ with $c(\mathcal C(A),x) = c_f(A,x)$ for every $x$. Taking the maximum over $x$, this implies $a_D(f) \ge sc_D(f)$.
	
	Second, consider a subcube partition $\mathcal C \sim f$ achieving the minimum \eqref{eq:def-det-sc-comp}. For every $x$, the fixed coordinates of $C(x)$ are a witness set $W_x(\mathcal C)$ for $f$ and $x$. Since $c(\mathcal C,x) = \abs{W_x(\mathcal C)} \ge w_f(x)$, taking the maximum over $x$ yields $sc_D(f) \ge w_D(f)$. 
	
	Third, fix any $x$ and consider a maximal collection of blocks $B_1,\ldots,B_m$ as in the definition of $b_f(x)$. We note that any witness set $W$ for $f$ and $x$ must intersect $B_i$ for every $1\le i\le m$. This implies $w_f(x) \ge m = b_f(x)$ and hence, $ w_D(f) \ge b_D(f)$. 
	
	Finally, for every $x$, the pivotal bits provide $s_f(x)$ disjoint blocks $B_1,\ldots,B_{s_f(x)}$ of size 1 satisfying \eqref{eq:block-pivotality}. This implies $s_f(x) \le b_f(x)$, and hence $s_D(f) \le b_D(f)$.
\end{proof}

The following examples show that the five complexity measures are all distinct. 

\begin{example}\label{ex:Nisan-function} For $n$ being a multiple of 4, define the Boolean function $g$ on $n$ bits by
	\begin{equation}
		g(x_1,\ldots,x_n) = \begin{cases} 	1 &\textrm{if}\ \sum_{i=1}^n x_i \in \{\frac{n}{2},\frac{n}{2}+1\}, \\
			0 &\textrm{otherwise}. \\
		\end{cases}
	\end{equation} 
	This function appeared in \cite{Nisan1989} to show that $w_D(g)>b_D(g)>s_D(g)$. It has deterministic complexities
	\begin{align}
		s_D(g) = \frac{n}{2} + 2,\ b_D(g) = \max\left\{\frac{3n}{4},\frac{n}{2}+2\right\},\ \textrm{and}\  w_D(g) = n-1. 
	\end{align}
	To see this, note that $s_g(x) = \frac{n}{2} + 2$ if $x$ has exactly $\frac{n}{2} + 2$ bits with value 1. Moreover, $b_g(x) = \frac{3n}{4}$ and $w_g(x) = n-1$ if $x$ has exactly $\frac{n}{2}$ bits with value 1. 
	
	We note that the following small proposition also implies  $a_D(g) = n$, which could otherwise be seen using an adversary argument. 
	\begin{proposition}\label{prop:Nisan-function}
		For Example \ref{ex:Nisan-function}, $sc_D(g)  = n$.
	\end{proposition}
	\begin{proof}Denote by $S_i$ the set containing all $x$ having exactly $i$ bits with value $1$. Towards a contradiction, assume that there exists a subcube partition $\mathcal C$ containing only subcubes fixing at most $n-1$ bits. Note that for any $x \in S_{n/2}$, $C(x)$ must fix all $n/2$ bits with value 1 and exactly $n/2 -1$ bits with value 0 since $g$ would otherwise not be constant on $C(x)$.  Therefore, $C(x) \neq C(y)$ for all $x,y \in S_{n/2}$ with $x \neq y$. However, note that every such cube $C(x)$ contains an element $x' \in S_{n/2 + 1}$. Since $\abs{S_{n/2+1}}<\abs{S_{n/2}}$, the cubes cannot be disjoint, a contradiction.
	\end{proof}
\end{example}

\begin{example}\label{ex:Savicky-function} Savický's function on $n=4$ bits, which was introduced in \cite{Savicky2000,Savicky2002} and which we call $\textrm{MAJ}_4$ according to \cite{Kothari2015}, is defined by
	\begin{equation}
		\textrm{MAJ}_4 (x_1,x_2,x_3,x_4) = \begin{cases} 	1 &\textrm{if}\ 2x_1 + \sum_{i=2}^4 x_i \ge 3, \\
			0 &\textrm{otherwise}. \\
		\end{cases}
	\end{equation} 
	Note that this is a majority function with the first bit being 
	decisive in the case of a tie. It is easy to see that the deterministic complexities are as follows:
	\begin{align}
		&s_D(\textrm{MAJ}_4) = b_D(\textrm{MAJ}_4) = w_D(\textrm{MAJ}_4) = sc_D(\textrm{MAJ}_4) = 3, \\
		&a_D(\textrm{MAJ}_4) = 4,
	\end{align}
	once one observes that 
	\begin{equation}
		\mathcal C :=\{(\star,1,1,1),(\star,0,0,0),(1,1,0,\star),(0,1,0,\star),(1,\star,1,0),(0,\star,1,0),(1,0,\star,1),(0,0,\star,1)\}
	\end{equation}
	is a subcube partition. 
\end{example}

\subsection{Distributional complexity measures}\label{subsec:defs-hierarchy-dist}

In this subsection, we define \emph{distributional} complexity measures. We denote the product measure on $\{0,1\}^n$ with marginal $p \in [0,1]$ by $\pi_p$. We begin with introducing the distributional analogues of the five deterministic complexity measures encountered in the previous subsection. The \emph{$\pi_p$-distributional algorithmic complexity} is defined as
\begin{equation} \label{eq:def-dist-algo-comp}
	a(f,\pi_p) = \min_A \ \mathbb E_{x \sim \pi_p} \left[c_f(A,x)\right],
\end{equation}
the \emph{$\pi_p$-distributional subcube partition complexity} as
\begin{equation} \label{eq:def-dist-sc-comp}
	sc(f,\pi_p) = \min_{\mathcal C \sim f}\ \mathbb E_{x \sim \pi_p} \left[ c(\mathcal C,x)\right],
\end{equation}
the \emph{$\pi_p$-distributional witness complexity} as
\begin{equation}
	w(f,\pi_p) = \mathbb E_{x \sim \pi_p} \left[  w_f(x) \right],
\end{equation}
the \emph{$\pi_p$-distributional block sensitivity} as
\begin{equation}
	b(f,\pi_p) = \mathbb E_{x \sim \pi_p} \left[b_f(x)\right],
\end{equation}
and the \emph{$\pi_p$-distributional sensitivity}  as
\begin{equation}
	s(f,\pi_p) = \mathbb E_{x \sim \pi_p} \left[  s_f(x) \right].
\end{equation}

The bits queried by an algorithm $A$ until $f$ is determined and the bits fixed by a subcube partition naturally form random subsets of $[n]$ which are measurable with respect to $x$. In the following, we also consider random sets $\mathcal{I} \subseteq [n]$ which are not necessarily measurable with respect to $x$. However, any random set is implicitly assumed to be defined on the \emph{same} probability space as $x$, and we continue to write $\mathbb E_{x \sim \pi_p}$ for the corresponding expectation. A random set $\mathcal{I}\subseteq [n]$ is called a \emph{witness set} for $f$, denoted $\mathcal{I}\sim f$, if $x\vert_\mathcal{I}$ almost surely determines $f$. Clearly, we have $w(f,\pi_p) = \min_{\mathcal I \sim f} \mathbb E_{x \sim \pi_p} [\abs{\mathcal I}]$.

We call a random set $\mathcal I \subseteq [n]$ \emph{local} if for every fixed $J \subseteq [n]$, 
\begin{equation}
	\sigma\left(\{\mathcal I = J\},(x_i:{i\in J})\right)\ \text{is independent of}\ \sigma\left(x_i: {i\notin J}\right).
\end{equation}
In words, conditional on the random set $\mathcal I$ being equal to $J$ and on the values of $x$ on $J$, the bits outside of $J$ are still independent Ber($p$)-distributed.
We are now in position to define a sixth distributional complexity measure, which has no deterministic analogue. The \emph{$\pi_p$-distributional local witness complexity} is defined as 
\begin{equation}
	\ell(f,\pi_p) = \min_{\mathcal I \sim f,\; \mathcal I\; \text{local}} \mathbb E_{x \sim \pi_p} \left[\abs{\mathcal I}\right].
\end{equation}

\begin{proposition} \label{prop:hierarchy-dist} For every Boolean function $f$ and for every $p \in [0,1]$, the distributional complexities satisfy
	\begin{equation}
		a(f,\pi_p) \ge sc(f,\pi_p) \ge \ell(f,\pi_p)\ge w(f,\pi_p) \ge b(f,\pi_p) \ge s(f,\pi_p).
	\end{equation}%
\end{proposition}

\begin{proof}
	The inequalities $a(f,\pi_p) \ge sc(f,\pi_p)$ and $w(f,\pi_p) \ge b(f,\pi_p) \ge s(f,\pi_p)$ can be obtained analogously to the deterministic case in Proposition \ref{prop:hierarchy-det}, by simply replacing taking the maximum over $x$ by expectations with respect to $\pi_p$. The inequality $\ell(f,\pi_p)\ge w(f,\pi_p)$ is immediate from the definition. Finally, we establish the inequality $sc(f,\pi_p) \ge \ell(f,\pi_p)$. Given a subcube partition $\mathcal C$, we write $\mathcal{I}_{\mathcal{C}}$ for the random set defined by $\mathcal{I}_{\mathcal{C}}(x):= I_{C(x)}$. It is straightforward to check that $\mathcal{I}_\mathcal{C}$ is local. As $\mathcal{C} \sim f$ implies $\mathcal{I}_{\mathcal{C}} \sim f$, the desired inequality follows.
\end{proof}

Examples \ref{ex:majority-3}, \ref{ex:constant-function} and \ref{ex:separation-local-subcube} show that the six complexity measures are all distinct. In Section  \ref{sec:asymp-separation}, we will study the stronger notion of asymptotic separation.

\begin{example}\label{ex:majority-3}
	The MAJORITY function on $n=3$ bits, defined by 
	\begin{equation}
		\textrm{MAJ}_3 (x_1,x_2,x_3) = \begin{cases} 	1 &\textrm{if} \sum_{i=1}^3 x_i \ge 2, \\
			0 &\textrm{if} \sum_{i=1}^3 x_i \le 1, \\
		\end{cases}
	\end{equation}
	has distributional complexities
	\begin{align}
		&s(\textrm{MAJ}_3,\pi_p) = 2- 2 p^3 - 2(1-p)^3, \\
		&b(\textrm{MAJ}_3,\pi_p) = 2 - p^3 - (1-p)^3, \\
		&w(\textrm{MAJ}_3,\pi_p) = 2, \\
		&sc(\textrm{MAJ}_3,\pi_p) = a(\textrm{MAJ}_3, \pi_p) = 2 + 2p(1-p). 
	\end{align}
	Since the notion of local witness complexity involves an infinite number of choices for the joint distribution of $(\mathcal I,x)$, its computation seems to be more difficult than the other distributional complexities. However, in some cases, this complexity can be computed explicitly as the following proposition demonstrates.
	\begin{proposition}\label{prop:majority-3}
		$\ell(\textrm{MAJ}_3,\pi_p) = 2 + 2p(1-p)$. 
	\end{proposition}
	\begin{proof}
		The upper bound follows directly from Proposition \ref{prop:hierarchy-dist} and the value of $sc(\textrm{MAJ}_3,\pi_p)$ given above. For the lower bound, let $\mathcal I$ be a local witness set for $\textrm{MAJ}_3$. We analyze the constraints on the joint distribution of $(\mathcal I, x)$ imposed by being a witness set and by locality. First, since $\mathcal I$ is a witness set for $\textrm{MAJ}_3$, we must have $\{i,j\} \subseteq \mathcal I$ for some $i \neq j \in [3]$ satisfying $x_i = x_j$. Thus, there are two possibilities for $\mathcal I$ if the bits are not all equal (e.g., $\mathcal I = \{1,2\}$ or $\mathcal I = \{1,2,3\}$ if $x=(1,1,0)$), and four possibilities for $\mathcal I$ if the bits are all equal. Second, since $\mathcal I$ is local, we have for every permutation $(i,j,k)$ of $(1,2,3)$ and every $z \in \{0,1\}$,
		\begin{equation}\label{eq:majority-3-local-1}
			\P_{x \sim \pi_p}\left[\mathcal I = \{i,j\}, x_i = x_j = z, x_k = 1\right] = \frac{p}{1-p} \cdot \P_{x \sim \pi_p}\left[\mathcal I = \{i,j\}, x_i = x_j = z, x_k = 0\right] .
		\end{equation}
		For abbreviation, we define 
		\begin{equation}
			\beta_{\{i,j\}}^z := \P_{x \sim \pi_p}\left[\mathcal I = \{i,j\}, x_i = x_j = z, x_k = 1-z \right] \in [0,p^{1+z}(1-p)^{2-z}], 
		\end{equation}
		where $(i,j,k)$ is a permutation of $(1,2,3)$ and $z \in \{0,1\}$, and $\beta^z := \beta_{\{1,2\}}^z + \beta_{\{2,3\}}^z + \beta_{\{1,3\}}^z$.
		The main observation is now that \eqref{eq:majority-3-local-1} implies
		\begin{equation}\label{eq:majority-3-local-2}
			\frac{p}{1-p}  \beta^1 = \P_{x \sim \pi_p}\left[ \abs{\mathcal I} = 2, x= (1,1,1)\right]  \le p^3 \iff \beta^1 \le p^2(1-p),
		\end{equation}
		and analogously, $\beta^0 \le p(1-p)^2$.
		
		Finally, we express the expected size of $\mathcal I$ in terms of $\beta^1$ and $\beta^0$.
		\begin{align}
			\mathbb E_{x \sim \pi_p} \left[\abs{\mathcal I}\right] &= 2 \beta^1 + 3 \left(3p^2(1-p) -\beta^1\right) + 2 \frac{p}{1-p}\beta^1 + 3\left(p^3 - \frac{p}{1-p}\beta^1 \right) \\
			&+ 2 \beta^0 + 3 \left(3p(1-p)^2 -\beta^0\right) + 2 \frac{1-p}{p}\beta^0 + 3\left((1-p)^3 - \frac{1-p}{p}\beta^0 \right)\\
			&= 3 - \left(\beta^1 + \frac{p}{1-p}\beta^1 \right) - \left(\beta^0 + \frac{1-p}{p}\beta^0 \right) = 3 - \frac{\beta^1}{1-p} - \frac{\beta^0}{p}.
		\end{align}
		Using the previously obtained upper bounds on $\beta^1$ and $\beta^0$, we conclude that 
		\begin{equation}
			\mathbb E_{x \sim \pi_p} \left[\abs{\mathcal I}\right] \ge 3 - p^2 - (1-p)^2 = 2 + 2p(1-p),
		\end{equation}
		and the lower bound on $\ell(\textrm{MAJ}_3,\pi_p)$ follows.		
	\end{proof}
\end{example}

It remains to separate $sc(\cdot,\pi_{p})$ and $a(\cdot,\pi_{p})$ as well as  $\ell(\cdot,\pi_{p})$ and $sc(\cdot,\pi_{p})$, which we will do in the following two examples.

\begin{example}\label{ex:constant-function}
	The ALL-EQUAL function on $n=3$ bits, defined by 
	\begin{equation}
		\textrm{A-EQ}_3 (x_1,x_2,x_3) = \begin{cases} 1 &\textrm{if}\ x_1 = x_2 = x_3, \\
			0 &\textrm{otherwise}, \\
		\end{cases}
	\end{equation}
	has distributional complexities
	\begin{align}
		&s(\textrm{A-EQ}_3,\pi_p) = 1 + 2p^3 + 2(1-p)^3, \\
		&b(\textrm{A-EQ}_3,\pi_p) = w(\textrm{A-EQ}_3,\pi_p) = \ell(\textrm{A-EQ}_3,\pi_p) = sc(\textrm{A-EQ}_3,\pi_p) = 2 + p^3 + (1-p)^3, \\
		&a(\textrm{A-EQ}_3, \pi_p) = 2 + p^2 + (1-p)^2. 
	\end{align}
	An optimal subcube partition is given by 
	\begin{equation}
		\mathcal C =\{(0,0,0),(1,0,\star),(\star,1,0),(0,\star,1),(1,1,1)\}. 
	\end{equation}
\end{example}

\begin{example}\label{ex:separation-local-subcube}
	The Boolean function $g$ on $n=4$ bits, defined by
	\begin{equation}
		g(x_1,\ldots,x_4) = \begin{cases} 	1 &\textrm{if}\ (x_1,\ldots,x_4) \in  \{(1,0,0,1),(0,0,0,1),(0,1,0,1),(0,1,1,0)\}, \\
			0 &\textrm{otherwise}, \\
		\end{cases}
	\end{equation} 
	has distributional complexities
	
	\begin{equation}
		s(g,\pi_{1/2}) = 3/2,\ b(g,\pi_{1/2}) = 9/4,\ \text{and}\ w(g,\pi_{1/2}) = 37/16.
	\end{equation}
	
	\begin{proposition}\label{prop:separation-local-subcube}
		For Example \ref{ex:separation-local-subcube}, $\ell(g,\pi_{1/2}) \le 21/8$ and $a(g,\pi_{1/2}) = sc(g,\pi_{1/2}) = 11/4$. 
	\end{proposition}
	\begin{proof}
		To see this, we first construct a local witness set of expected size $21/8$. Define 
		\begin{equation}
			\mathcal{I}_0(x_1,\ldots,x_4) = \begin{cases} 	\{1,2\} &\textrm{if}\ x_1 = x_2 = 1, \\
				\{2,3\} &\textrm{if}\ x_2 = 0,\ x_3 = 1, \\
				\{1,2,3,4\} &\textrm{if}\ x_1 = 0,\ x_2 = 1, \\
				\{2,3,4\} &\textrm{if}\ x_2 = x_3 = 0,\ x_4 = 1, \\
				\{2,4\} & \textrm{if}\ x_2 = x_3 = x_4 = 0, \\
			\end{cases}
		\end{equation} 
		and 
		\begin{equation}
			\mathcal{I}_1(x_1,\ldots,x_4) = \begin{cases} 	\{3,4\} &\textrm{if}\ x_3 = x_4, \\
				\{1,2,3,4\} &\textrm{if}\ x_2 = 1,\ x_3 \neq x_4, \\
				\{2,3,4\} &\textrm{if}\ x_2 = x_3 = 0,\ x_4 = 1, \\
				\{2,4\} & \textrm{if}\ x_2 = x_4 = 0,\ x_3 = 1. \\
			\end{cases}
		\end{equation}
		It is easy to see that $\mathcal{I}_0$ and $\mathcal{I}_1$ are witness sets of size $\mathbb{E}_{x\sim \pi_{1/2}}[\abs{\mathcal{I}_0}]=\mathbb{E}_{x\sim \pi_{1/2}}[\abs{\mathcal{I}_1}]=21/8$, but not local. However, for a Ber($1/2$)-distributed random variable $G$ that is independent of $x$, one can check that the random set $\mathcal{I}_G$ is a local witness set.
		
		Second, we describe an algorithm whose expected number of queries is $11/4$. The algorithm begins with querying $x_3$ and $x_4$. Note that $f=0$ if $x_3 = x_4$. The algorithm then queries $x_1$. Note that $f=0$ (resp. $f=1$)  if $(x_1,x_3,x_4) = (1,1,0)$ (resp.\ $(0,0,1)$). Lastly, the algorithm queries $x_2$.
		
		Third, we need to argue that every subcube partition determining $f$ fixes on average at least $11/4$ bits. We study the subcube partitions of $\{x:f(x)=1\}$ and $\{x:f(x)=0\}$ separately. It is easy to see that $\{(\star,0,0,1),(0,1,0,1),(0,1,1,0)\}$  is an optimal subcube partition of $\{x:f(x)=1\}$, and it fixes on average $7/2$ bits (conditional on $f=1$). Since $11/4 = (1/4) \cdot (7/2) + (3/4)\cdot (5/2)$, we are thus left with showing that an optimal subcube partition of $\{x:f(x)=0\}$ fixes on average at least $5/2$ bits (conditional on $f=0$). We first observe that every subcube in $\{x:f(x)=0\}$ must fix at least 2 bits, i.e.\ has size at most 4. Next, if a subcube partition of $\{x:f(x)=0\}$ contains only one (resp.\ no) subcube of size 4, it fixes on average at least $8/3$ (resp.\ $3$) bits, which is larger than $5/2$. Otherwise, the subcube partition of $\{x:f(x)=0\}$ contains at least two subcubes of size 4. A case-by-case analysis allows to verify that for all possible combinations of two subcubes of size 4, the remaining 4 configurations can neither be combined into one subcube of size 4 nor into two subcubes of size 2. This establishes that any subcube partition of $\{x:f(x)=0\}$ fixes on average at least $(2\cdot 8 + 3 \cdot 2 + 4 \cdot 2)/12 = 5/2$ bits, and thereby completes the argument.
	\end{proof}
\end{example}

We note that for any Boolean function $f$, the separation $s_D(f) < b_D(f) < w_D(f)$ directly implies the separation $s(f,\pi_p) < b(f,\pi_p) < w(f,\pi_p)$ for $p \in (0,1)$ since these three complexity measures are ordered pointwise for every $x$. However, the converse is generally not true which can be seen, for example, by considering the majority function on three bits satisfying $s(\textrm{MAJ}_3,\pi_p) < b(\textrm{MAJ}_3,\pi_{p}) < w(\textrm{MAJ}_3,\pi_{p})$ but $s_D(\textrm{MAJ}_3) = b_D(\textrm{MAJ}_3) = w_D(\textrm{MAJ}_3)=2$.

We also note that the separation $sc_D(\textrm{MAJ}_4) < a_D(\textrm{MAJ}_4) $ in Example \ref{ex:Savicky-function} does not imply the corresponding separation in the distributional case, and it can be checked that $sc(\textrm{MAJ}_4,\pi_{1/2}) = a(\textrm{MAJ}_4,\pi_{1/2}) $.

We will see in Proposition \ref{prop:witness-block-equality} that distributional witness complexity and block sensitivity are equal for a large class of functions. However, asking for equality of distributional sensitivity and block sensitivity severely restricts which Boolean function we have and we will prove below that this equality implies that the function is a parity function (or negative parity) on some collection of the bits. 
Note that  $s(f,\pi_p)=b(f,\pi_p)$ implies that the sensitivity and block
sensitivity agree on every bit string which in term implies that every bit
string has a pivotal bit (unless f is  constant). The last thing certainly does not imply we are a parity function on some subset of the bits since we can take any Boolean function and XOR it with one bit, guaranteeing at least one pivotal bit.

\begin{proposition}\label{prop:s-b-equal}
	Let $f$ be a Boolean function and $p \in (0,1)$. The following are equivalent: 
	\begin{itemize}
		\item[(i)] $f$ or $1-f$ is the PARITY function on some subset of the bits,
		\item[(ii)] $b(f,\pi_{p}) = s(f,\pi_{p})$.
	\end{itemize}
\end{proposition}
In particular, $b(f,\pi_{p}) = s(f,\pi_{p})$ implies that all distributional complexity measures for $f$ are equal. It would be interesting to obtain approximate versions of this result. However, note that for the MAJORITY function on $n$ bits, the distributional sensitivity and block sensitivity differ only by a multiplicative factor $\log(n)$ whereas the distributional sensitivity and witness complexity differ by a multiplicative factor $\sqrt{n}$ (see Example \ref{ex:majority-n}).
\begin{proof}
	The implication $(i) \implies (ii)$ being trivial, we only need to prove $(ii) \implies (i)$. Without loss of generality, fix a Boolean function $f$ that depends on every bit and let $n$ denote the number of bits. First, since $b_f(x) \ge s_f(x)$ for every $x \in \{0,1\}^n$, the equality in $(ii)$ implies that for every $x \in \{0,1\}^n$,
	\begin{equation}
		b_f(x) = s_f(x).
	\end{equation}
	We remind the reader that a Boolean function $f$ can naturally be viewed as a site percolation configuration on the hypercube $\{0,1\}^n$. Denote by $D(x)$ the connected component of $x$ and by $S_f(x)$ the set of pivotal bits for $f$ at $x$. Now, let us argue that for every $x \in \{0,1\}^n$, 
	\begin{equation}\label{eq:1-star}
		D(x) =  \{y \in \{0,1\}^n : y_i = x_i, \forall i \in S_f(x) \} =: \widetilde{D}(x).
	\end{equation}
	To this end, we fix any $x\in \{0,1\}^n$. Since $s_f(x) =b_f(x)$, we must have $f(x^B) = f(x)$ for every $B \subseteq [n]\setminus S_f(x)$, and so it follows that $ \widetilde{D}(x) \subseteq D(x)$. Note that this implies $S_f(y) \subseteq S_f(x)$ for every $y \in  \widetilde{D}(x)$. Towards a contradiction, assume that $j \in S_f(x) \setminus S_f(y)$. Then the block 
	\begin{equation}
		B := \{j\} \cup \{i \in [n] \setminus S_f(x): x_i \neq y_i\}
	\end{equation}
	satisfies $f(y) = f(x) \neq f(x^j) = f(y^B)$ and $B \cap S_f(y) = \emptyset$, which contradicts $b_f(y) = s_f(y)$. Thus, $S_f(y) = S_f(x)$ for every $y \in  \widetilde{D}(x)$ and so \eqref{eq:1} follows. In other words, we have established that every connected component is a subcube.
	
	To conclude from here, we assume without loss of generality that $f(0)=0$ and  define for every $I \subseteq S_f(0)$,
	\begin{equation}
		D_I := \{y \in \{0,1\}^n : y_i = 0, \forall  i \in S_f(x)\setminus I,\ \text{and}\ y_i =1, \forall i \in I\}.
	\end{equation}
	Note that $D_\emptyset = D(0)$ and that the graph distance between  $D_I$ and $D_J$ is equal to $\abs{I \triangle J}$. Using \eqref{eq:1-star}, it is now straightforward to check by induction on $\abs{I}$ that 
	\begin{equation}\label{eq:2-star}
		f\vert_{D_I} \equiv \abs{I} \mod 2, \quad \forall I \subseteq S_f(0).
	\end{equation}
	Hence, $f$ does not depend on the bits in $[n] \setminus S_f(0)$, which implies that $S_f(0) = [n]$. Combining this with \eqref{eq:2-star}, we conclude that $f$ is the PARITY function. 
\end{proof}

For monotone Boolean functions, the next proposition shows that asking for equality of distributional witness complexity and subcube partition complexity implies that the function is a dictator function on one of its bits. It would be interesting to know if equality of distributional witness complexity and local witness complexity is sufficient.
\begin{proposition}\label{prop:sc-w-equal}
	Let $f$ be a monotone, non-constant Boolean function and let $p \in (0,1)$. The following are equivalent: 
	\begin{itemize}
		\item[(i)] $f$ is a DICTATOR function (i.e.\ $f(x)=x_i$ for some $i\in [n]$),
		\item[(ii)] $sc(f,\pi_{p}) = w(f,\pi_{p})$.
	\end{itemize}
\end{proposition}
\begin{proof}
	The implication $(i) \implies (ii)$ being trivial, we proceed with proving $(ii) \implies (i)$. Assume that $f$ is not a DICTATOR function. Fix an arbitrary subcube partition $\mathcal C \sim f$, and denote by $C_0$ (resp.\ $C_1$) the subcube containing $(0,\ldots,0)$ (resp.\ $(1,\ldots,1)$). Since $f$ is monotone and non-constant, we have $f\vert_{C_0} \equiv 0$  and  $f\vert_{C_1} \equiv 1$. Since $f$ is not a DICTATOR function, it is easy to see that $\mathcal C \neq \{C_0,C_1\}$. Now, let $C' \in \mathcal C$ be a subcube different from $C_0$ and $C_1$, and assume without loss of generality that $f\vert_{C'}\equiv 0$. Since $C'\cap C_0 = \emptyset$, the set $I_{C' }$ of bits fixed by $C'$ must contain some bit $i'$ that is fixed to be $1$. However, by monotonicity of $f$, the set $I_{C'}\setminus \{i'\}$ is a witness set for $f$ and $x \in C'$. This establishes that $sc(f,\pi_p) > w(f,\pi_p)$. 
\end{proof}
\begin{question}\label{question:monotone-subcube-algorithmic}
	Does $a(f,\pi_p)=sc(f,\pi_p)$ hold for every monotone Boolean function $f$ and for every $p \in (0,1)$?
\end{question}

The following standard example shows why monotonicity is needed in the above proposition. 
\begin{example} \label{ex:address}
	The ADDRESS function on $n=m + 2^m$ bits, defined by 
	\begin{equation}
		\textrm{ADDRESS}_m(x_1,\ldots,x_m,y_0,\ldots,y_{2^m-1}) = y_{\sum_{i=0}^{m-1} x_i 2^i},
	\end{equation}
	has distributional complexities
	\begin{align}
		&w(\textrm{ADDRESS}_m,\pi_p)  = a(\textrm{ADDRESS}_m,\pi_p) = m+1.
	\end{align}
\end{example}

\subsection{A geometric perspective on subcube partitions and algorithms} \label{subsec:subcube}

In this subsection, we follow a more geometric approach. One of the first things we want to understand is when a particular subcube partition arises from an algorithm. We make the following observation: If a subcube partition arises from an algorithm, then it must be the case that (1) there is some hyperplane orthogonal to one of the coordinate directions 
that splits the cube in half and  crosses no edge with endpoints belonging to the same subcube and (2) each of the two subcube partitions induced by this splitting arises from an algorithm. 
Therefore, Figure \ref{fig:geometric-illustrations-a} illustrates a subcube partition that does not arise from an algorithm.
\begin{figure}
	\centering
	\begin{minipage}{.4\textwidth}
		\centering
		\includegraphics[width=\textwidth]{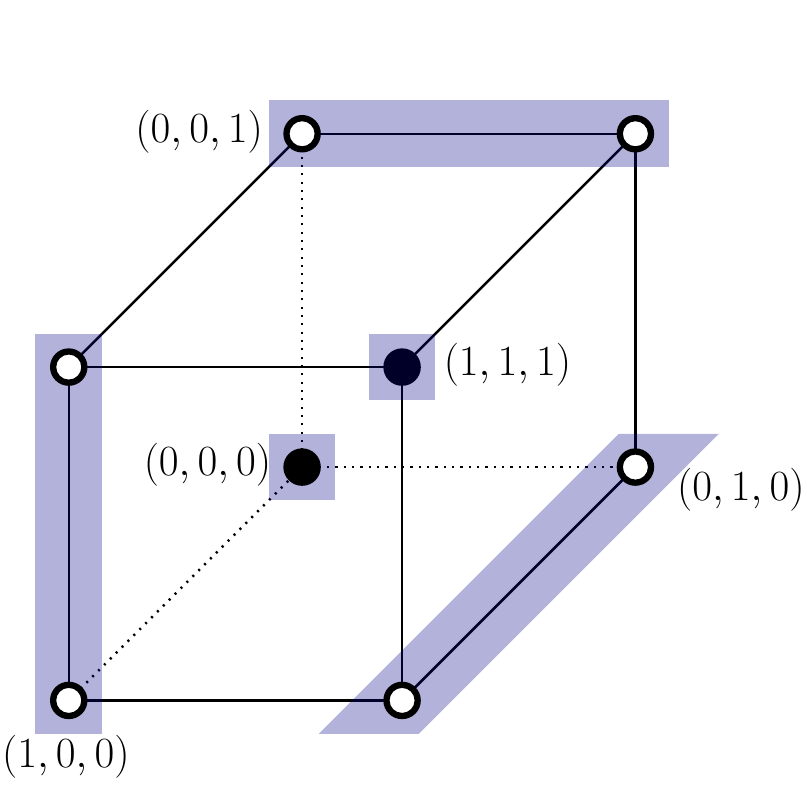}
		\subcaption{}\label{fig:geometric-illustrations-a}
	\end{minipage}%
	\hfill
	\begin{minipage}{.5\textwidth}
		\centering
		\includegraphics[width=\textwidth]{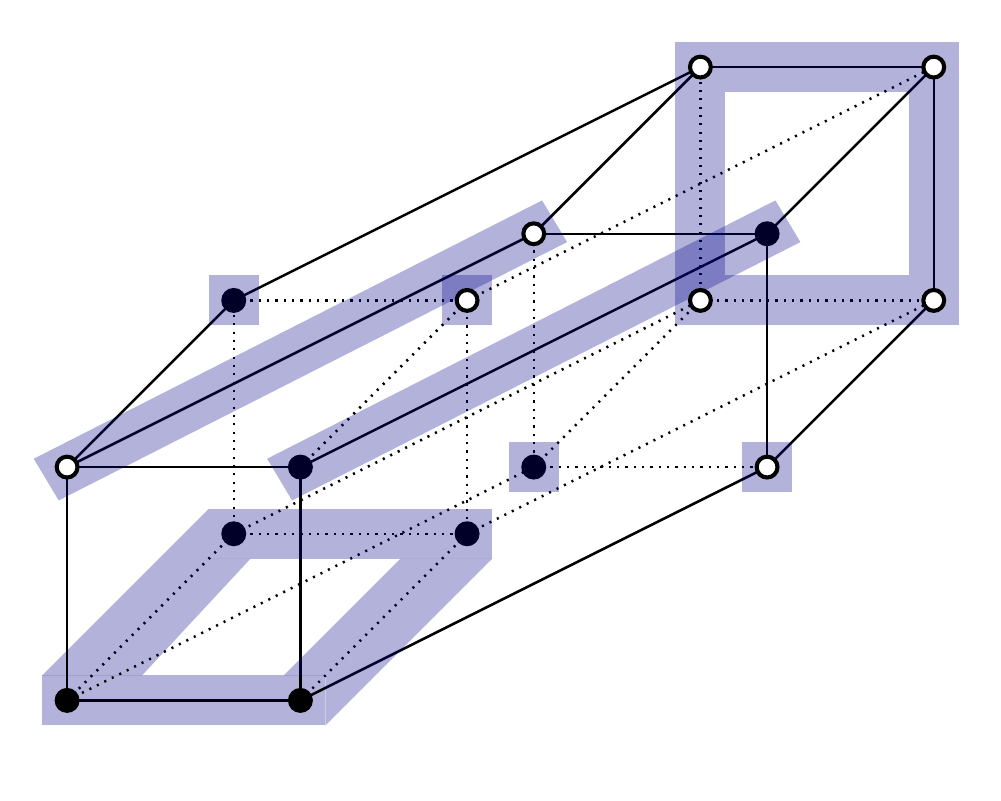}
		\subcaption{}\label{fig:geometric-illustrations-b}
	\end{minipage}
	\caption{Illustrations of (a) the ALL-EQUAL function on $n=3$ bits (see Example \ref{ex:constant-function}) and (b) the function $h$ on $n=4$ bits (see Equation \eqref{eq:function-seperating-sc-a}) by marking inputs with function value $1$ (resp.\ $0$) using a disk with black (resp.\ white) filling. In (b), the first three directions are as in (a) and additionally the three-dimensional cube in the front (resp.\ back) contains the inputs with $x_4 = 1$ (resp. $0$). 
		The illustrations also show optimal subcube partitions for the respective functions with respect to $\pi_{1/2}$.} 
\end{figure}

It follows that in order to show that $sc(f,\pi_p)<a(f,\pi_p)$ for a given Boolean function $f$, one needs to show that no optimal subcube partition with respect to $\pi_p$ determining $f$ arises from an algorithm. Since it is easy to verify that there are only two optimal subcube partitions for the ALL-EQUAL function on $n=3$ bits from Example \ref{ex:constant-function} (see Figure \ref{fig:geometric-illustrations-a} for a representation of the optimal subcube partition $\{(0,0,0),(1,0,\star),(\star,1,0),(0,\star,1),(1,1,1)\}$; the second one being essentially the same), we see that $sc(\textrm{A-EQ}_3,\pi_{p}))<a(\textrm{A-EQ}_3,\pi_{p})$.

We now point out an interesting difference between subcube partitions and algorithms for Boolean functions.
Any subcube partition of $\{0,1\}^n$ which determines a Boolean function $f$ induces in a natural way a
subcube partition of the subset $f^{-1}(1)$ and a subcube partition of the subset $f^{-1}(0)$. Moreover, it is clear that
when trying to find an optimal subcube partition for a Boolean function $f$ (either deterministically in worst  case or
distributionally with respect to some $\pi_p$), one can work separately on $f^{-1}(1)$ and $f^{-1}(0)$, finding
optimal solutions for the two parts separately and then just combining them. When we move to algorithms, while
it does make sense to talk about optimal algorithms on $f^{-1}(1)$ and $f^{-1}(0)$ separately, one cannot \emph{combine}
such algorithms into an algorithm for the whole cube in the same way.
(For distributional complexity, optimal algorithms on $f^{-1}(1)$ and $f^{-1}(0)$
means precisely finding algorithms which minimize the expected number of
queries made conditional on $f^{-1}(1)$ and $f^{-1}(0)$ respectively.)
The following interesting example on $4$ bits illustrates such an example.

Define the Boolean function $h$ on $n=4$ bits by
\begin{equation}\label{eq:function-seperating-sc-a}
	h(x_1,\ldots,x_4) = \begin{cases} 	1 &\textrm{if}\ (x_1,\ldots,x_4) \in  \{(\star,\star,0,1),(1,1,1,\star),(1,0,0,0),(0,0,1,1)\}, \\
		0 &\textrm{otherwise}. \\
	\end{cases}
\end{equation} 
We refer the reader to Figure \ref{fig:geometric-illustrations-b} for an illustration of $h$ and the unique optimal subcube partition with respect to $\pi_{1/2}$. We leave it to the reader to check  that $sc(h,\pi_{1/2})= 11/4$ and that $a(h,\pi_{1/2})= 3> 11/4$, the latter being obtained by a case-by-case analysis. However, if we condition on $h^{-1}(1)$ or $h^{-1}(0)$, one can check that the conditional subcube partition complexities are still $11/4$ but more interestingly, the optimal algorithms on the two pieces also achieve $11/4$.
On $h^{-1}(1)$, one sees this by taking the algorithm which first queries $x_3$, second $x_1$ (resp.\ $x_4$) if $x_3 =1$ (resp.\ $0$) and then continues in the obvious way. On $h^{-1}(0)$, one sees this by taking the same algorithm but with the roles of $x_1$ and $x_3$ being switched.
However, there is no meaningful way to combine these algorithms for the whole cube. Note that the optimal algorithm on $h^{-1}(1)$ above yields an (non-optimal) expected number of queries of $13/4$ on
$h^{-1}(0)$ and vice versa. 
In contrast, for the ALL-EQUAL function, on  $\textrm{A-EQ}_3^{-1}(0)$, the optimal subcube cube and algorithmic costs differ (but are equal on $\textrm{A-EQ}_3^{-1}(1)$).

\subsection{A stopping set perspective on subcube partitions and local sets}\label{subsec:stopping-set-perspective}

In the proof of Proposition \ref{prop:hierarchy-dist}, we have already seen that any subcube partition naturally induces a local set, and this allowed us to establish the comparison $sc(f,\pi_p) \ge \ell(f,\pi_p)$. Here, we develop a unifying perspective on subcube partitions and local sets. 

For the following discussion, we fix a general probability space $(\Omega,\F, \P)$ on which the bits $x \sim \pi_p$ are defined. A collection $\mathbb F := (\F_J)_{J\subseteq [n]}$ of sub-$\sigma$-algebras of $\F$ is a \emph{filtration} if $\F_{J_1} \subseteq \F_{J_2}$ for all $J_1 \subseteq J_2$, and a random set $\mathcal I$ is a \emph{stopping set with respect to $\mathbb F$} if $\{\mathcal I = J \} \in \F_J$ for every fixed $J\subseteq [n]$. In addition, we call $x$ an \emph{independent field with respect to $\mathbb F$} if for every fixed $J \subseteq [n]$, 
\begin{equation}
	\sigma\left(x_i : {i\in J}\right) \subseteq \mathcal F_J,\ \text{and}\  \sigma\left(x_i: {i\notin J}\right) \text{is independent of}\ \F_J.
\end{equation}

\begin{proposition}\label{prop:stopping-set-local-set}
	Fix $(\Omega,\F, \P)$ on which the bits $x \sim \pi_p$ are defined. For a random set $\mathcal I \subseteq [n]$, the following are equivalent:
	\begin{enumerate}
		\item[(i)] There exists a filtration $\mathbb F$ such that $\mathcal I$ is a stopping set with respect to $\mathbb F$  and $x$ is an independent field with respect to $\mathbb F$,
		\item[(ii)] $\mathcal{I}$ is local.
	\end{enumerate}
\end{proposition}

\begin{proof}
	The implication $(i) \implies (ii)$ is straightforward from the definitions. To establish the converse, we define the filtration $\mathbb F := (\F_J)_{J\subseteq [n]}$ by 
	\begin{equation}
		\F_J := \sigma\left((\{\mathcal I = J'\}: J' \subseteq J),(x_i : i \in J)\right). 
	\end{equation}
	$\mathcal{I}$ is clearly a stopping set with respect to $\mathbb F$, and we are left with showing that $x$ is an independent field with respect to $\mathbb F$. First, $\sigma(x_i : i \in J) \subseteq \F_J$ follows from the definition. Second, we fix any $A \in \F_J$ and $B \in \sigma(x_i: {i\notin J})$ and, using locality, aim to show $\mathbb P[A \cap B] = \mathbb P[A]\cdot \mathbb P[B]$. It is possible to choose, for $J' \subseteq J$, sets $\mathcal Y(J',A) \subseteq \{0,1\}^J$ and $\mathcal{Y}^\ast(A) \subseteq \{0,1\}^J$ such that
	\begin{equation}
		A = \left(\bigsqcup_{J' \subseteq J}\ \bigsqcup_{y\vert_J \in \mathcal{Y}(J',A)} \{\mathcal{I} = J'\} \cap \{x\vert_J = y\vert_J\}\right) \sqcup
		\left(\bigsqcup_{y\vert_J \in \mathcal{Y}^\ast(A)} \{\mathcal{I} \not\subseteq J\} \cap \{x\vert_J = y\vert_J\}\right),
	\end{equation}
	and $\mathcal Y(B) \subseteq \{0,1\}^{J^\text{c}}$ such that
	\begin{equation}
		B = \bigsqcup_{y\vert_{J^{\text{c}}} \in \mathcal{Y}(B)} \left\{x\vert_{J^{\text{c}}} = y\vert_{J^{\text{c}}}\right\}.
	\end{equation}
	Using these partitions of $A$ and $B$, we get
	\begin{align}
		\P[A\cap B] &= \sum_{J':J' \subseteq J}\ \sum_{y\vert_J \in \mathcal{Y}(J',A)}\ \sum_{y\vert_{J^{\text{c}}} \in \mathcal{Y}(B)} \P\left[\{\mathcal I = J'\} \cap \{x = y\}\right] \\ 
		&+ 
		\sum_{y\vert_J \in \mathcal{Y}^\ast(A)}\ \sum_{y\vert_{J^{\text{c}}} \in \mathcal{Y}(B)} \P\left[\{\mathcal I \not\subseteq J\} \cap \{x = y\}\right]. \label{eq:prop-local-1}
	\end{align}
	Now, for every $J' \subseteq J$ and $y \in \{0,1\}^n$, 
	\begin{align}
		&\P\left[\{\mathcal I = J'\} \cap \{x = y\}\right] = \P\left[\{\mathcal I = J'\} \cap \{x\vert_{J'} = y\vert_{J'}\}\right] \cdot\P\left[x\vert_{(J')^\text{c}} = y\vert_{(J')^\text{c}}\right] \\
		&= \P\left[\{\mathcal I = J'\} \cap \{x\vert_{J'} = y\vert_{J'}\}\right] \cdot\P\left[x\vert_{J\setminus J'} = y\vert_{J\setminus J'}\right] \cdot\P\left[x\vert_{J^\text{c}} = y\vert_{J^\text{c}}\right] \\
		&= \P\left[\{\mathcal I = J'\} \cap \{x\vert_{J} = y\vert_{J}\}\right] \cdot\P\left[x\vert_{J^\text{c}} = y\vert_{J^\text{c}}\right],
	\end{align}
	where we used in the first and third equality that $\mathcal{I}$ is local. Using the above, one can also show that $\P[\{\mathcal I \not\subseteq J\} \cap \{x = y\}] = \P[\{\mathcal I \not\subseteq J\} \cap \{x\vert_J = y\vert_J\}]  \cdot\P\left[x\vert_{J^\text{c}} = y\vert_{J^\text{c}}\right]$. Plugging these back into \eqref{eq:prop-local-1}, one deduces that $\P[A\cap B] = \P[A]\cdot \P[B]$ and this concludes the proof.
\end{proof}
The following simple proposition is left to the reader.
\begin{proposition}\label{prop:subcube-stopping-sets}
	There is a natural bijection between subcube partitions and stopping sets with respect to the canonical filtration $(\sigma(x_i : i \in J))_{J \subseteq [n]}$.
\end{proposition}
Note that $x$ is an independent field with respect to the canonical filtration since the bits are independent. The two previous propositions shed light on the relation between subcube partitions and local sets by viewing the two concepts from the unifying perspective of stopping sets and independent fields. 

We conclude this subsection with a brief discussion of randomized subcube partitions. We recall for a subcube partition $\mathcal C$, the random set $\mathcal{I}_\mathcal{C}$, defined by $\mathcal{I}_\mathcal{C}(x)=I_{C(x)}$, is a stopping set with respect to the canonical filtration.
A randomized subcube partition is a probability measure $\mu$ on the space of subcube partitions. If $G \sim \mu$ is independent of $x$, we observe that $\mathcal I_G$ is a stopping set with respect to the filtration
\begin{equation}
	\left(\sigma\left((x_i:i \in J),G\right)\right)_{J \subseteq [n]},
\end{equation}
and $x$ is an independent field with respect to this filtration. Thus, $\mathcal{I}_G$ is a local set by Proposition \ref{prop:stopping-set-local-set}. This particular local set has a ``product form'' and Example \ref{ex:separation-local-subcube} shows that \emph{not} every local set can be constructed in this way since, in this case, the expected size of the local set is strictly smaller than the distributional subcube partition complexity. 

Going the other way around, let $\mathcal G$ be a $\sigma$-algebra that is independent of $\sigma(x_i: i\in [n])$. If we are now given a stopping set $\mathcal I$ with respect to the filtration 
\begin{equation}
	\left(\sigma\left((x_i:i \in J),\mathcal G\right)\right)_{J \subseteq [n]},
\end{equation}	
then we have that $x$ is an independent field with respect to this filtration and $\mathcal I$ naturally gives rise to a randomized subcube partition.

\section{Boolean function composition}\label{sec:composition}

In this section, we analyze the behavior of complexity measures under compositions of Boolean functions. This will be very useful to bound or even compute  complexities of more complicated Boolean functions which are constructed using compositions.
For two Boolean functions $f$, $g$ on $n$ respectively $m$ bits, the \emph{composition} $f \circ g$ is the Boolean function on $n\cdot m$ bits  defined by
\begin{equation*}
	f\circ g (x^1_{1},\ldots,x^1_{m},\ldots,x^n_{1},\ldots,x^n_{m}) = f\left(g(x^1_{1},\ldots,x^1_{m}),\ldots,g(x^n_{1},\ldots,x^n_{m})\right).
\end{equation*} 
\subsection{Review of the deterministic case}\label{subsec:composition-det}

Boolean function composition is well-understood with respect to deterministic complexity measures.
\begin{proposition}\label{prop:composition-det}
	For all Boolean functions  $f$ and $g$,
	\begin{enumerate}
		\item[(i)] $s_D(f \circ g) \le s_D(f) \cdot s_D(g)$,
		\item[(ii)] $w_D(f \circ g) \le w_D(f) \cdot w_D(g)$,
		\item[(iii)] $sc_D(f \circ g) \le sc_D(f) \cdot sc_D(g)$,
		\item[(iv)] $a_D(f \circ g) = a_D(f) \cdot a_D(g)$.
	\end{enumerate}
\end{proposition}
We refer to \cite[Lemma 3.1]{Tal2013} for a proof of (i), (ii) and (iv), and to \cite[Prop.\ 3]{Kothari2015} for a proof of (iii). 
The reader will immediately notice that block sensitivity is missing in the above proposition. While one might guess that $b_D(f\circ g) \le b_D(f) \cdot b_D(g)$, it was observed in \cite[Example 5.8]{Tal2013} that this inequality is false. The following simple example shows that the first two inequalities can be strict. In the case of subcube partition complexity, we are not aware of an example showing that the inequality in (iii) can be strict.

\begin{example}\label{ex:tribes-function}
	The OR-function on $n$ bits is defined by
	\begin{equation}
		\textrm{OR}_n (x_1,\ldots,x_n) = 
		\begin{cases} 	1 &\textrm{if}\  x_i = 1 \ \textrm{for some}\ i \in [n], \\
			0 &\textrm{otherwise}. \\
		\end{cases}
	\end{equation}The AND-function on $n$ bits is defined by
	\begin{equation}
		\textrm{AND}_n (x_1,\ldots,x_n) = 
		\begin{cases} 	1 &\textrm{if}\  x_i = 1 \ \textrm{for all}\ i \in [n], \\
			0 &\textrm{otherwise}. \\
		\end{cases}
	\end{equation}
	It is easy to see that all deterministic complexity measures of $\textrm{OR}_n$ and of $\textrm{AND}_n$ are equal to $n$. 
	Now, let $f$ be the OR-function on $\ell$ bits and let $g$ be the AND-function on $m$ bits. 
	Then $f\circ g$ is the well-known TRIBES-function with $\ell$ tribes of size $m$ on $n=\ell m$ bits. 
	The deterministic complexity measures are given by
	\begin{equation}\label{ex:composition-det}
		a_D(f\circ g) = sc_D(f\circ g) = n	\ \textrm{and} \ s_D(f\circ g) = b_D(f\circ g) = w_D(f\circ g) = \max\{\ell,m\}.
	\end{equation}
	These are easily justified, except perhaps $sc_D(f\circ g)$ which can be argued as follows: By \cite[Lemma 1]{Rivest1975}, for any Boolean function $h$, if $\{x:h(x)=1\}$ is odd, then the deterministic subcube partition complexity equals the number of bits. For the TRIBES function $f \circ g$ as above, it is easy to see that 
	\begin{equation}
		\abs{\{x : f \circ g(x) = 1\}} = \sum_{k=1}^\ell (-1)^{k+1} \binom{\ell}{k} \left(2^m\right)^{\ell - k},
	\end{equation}
	which is odd since the summands corresponding to $k=1,\ldots,\ell -1$ are even and the summand corresponding to $k=\ell$ is odd.
\end{example}

\subsection{Composition results for distributional complexities}\label{subsec:composition-dist}

The behavior of distributional complexity measures under Boolean function composition differs considerably from the deterministic case as the next result summarizes. 

\begin{proposition}\label{prop:composition-distributional}
	Set $g(p) := \P_{x \sim \pi_p}[g(x)=1]$. Then for all Boolean functions  $f$ and $g$,
	\begin{enumerate}
		\item[(i)] 	$s(f\circ g,\pi_p) = s(f,\pi_{g(p)}) \cdot s(g,\pi_p)$,
		\item[(ii)] $	\ell(f\circ g,\pi_p) \le \ell(f,\pi_{g(p)}) \cdot \ell(g,\pi_p)$,
		\item[(iii)] $	sc(f\circ g,\pi_p) \le sc(f,\pi_{g(p)}) \cdot sc(g,\pi_p)$,
		\item[(iv)] $	a(f\circ g,\pi_p) \le a(f,\pi_{g(p)}) \cdot a(g,\pi_p)$.
	\end{enumerate}
\end{proposition}

These results are proven at the end of this subsection. We will often make use of the following immediate corollary.

\begin{corollary}\label{cor:composition-distributional}
	If a Boolean function $f$ satisfies $\P_{x \sim \pi_p}[f(x)=1]=p$, then for every $k\ge 1$,
	\begin{enumerate}
		\item[(i)] 	$s(f^k,\pi_p) = s(f,\pi_p)^k$,
		\item[(ii)] $	\ell(f^k,\pi_p) \le \ell(f,\pi_p)^k$,
		\item[(iii)] $	sc(f^k,\pi_p) \le sc(f,\pi_p)^k$,
		\item[(iv)] $	a(f^k,\pi_p) \le a(f,\pi_p)^k$,
	\end{enumerate}
\end{corollary}

The previous proposition makes no statement about the distributional block sensitivity or witness complexity of composed Boolean functions. The next example in fact shows that submultiplicativity does not necessarily hold for these two complexities.

\begin{example}
	Choose $f=\textrm{OR}_2$ and $g=\textrm{AND}_2$ so that $f \circ g$ is the TRIBES-function with $\ell = m=2$ (see also Example \ref{ex:composition-det}). On the one hand, it immediately follows from the formulas in Example \ref{ex:or-function} (for the OR function, the complexities are the same as for the AND function with $p$ replaced by $1-p$) and $g(p)=p^2$  that 
	\begin{equation}
		b(f,\pi_{g(p)}) \cdot b(g,\pi_{p}) = w(f,\pi_{g(p)}) \cdot w(g,\pi_{p}) = (1+(1-p^2)^2) \cdot (1+p^2).
	\end{equation}
	On the other hand, it is easy to see that $b_{f \circ g}(x) = w_{f \circ g}(x) = 2$ for every $x\in \{0,1\}^4$ and thus,
	\begin{equation}
		b(f\circ g,\pi_p) = w(f\circ g,\pi_p) = 2.
	\end{equation}
	Since $2 > (1+(1-p^2)^2) \cdot (1+p^2)$ for every $p \in (0,1)$, this example shows that neither the distributional block sensitivity nor the witness complexity of $f \circ g$ can be bounded from above by the product of the respective complexities of $f$ and $g$. 
\end{example}
Not surprisingly, a general lower bound in terms of the product of the respective complexities cannot hold either for distributional block sensitivity and witness complexity. This can easily be checked by considering $f=g=\textrm{AND}_2$.

\begin{remark}
	It is no coincidence that the distributional block sensitivity and distributional witness complexity are equal in the previous example. We will get back to this in Subsection \ref{subsec:percolation-examples}.
\end{remark}

Finally, let us mention that the inequalities in parts (ii)--(iv) of Proposition \ref{prop:composition-distributional} can be strict. This can be seen by taking $f=g=\textrm{MAJ}_3$, for which we have seen in Example \ref{ex:majority-3} that $\ell(\textrm{MAJ}_3,\pi_{1/2}) = sc(\textrm{MAJ}_3,\pi_{1/2}) = a(\textrm{MAJ}_3, \pi_{1/2}) = 5/2$. It is well-known and can easily be checked that $a(f \circ g, \pi_{1/2}) < (5/2)^2$.

\begin{proof}[Proof of Proposition \ref{prop:composition-distributional}] 
	Let $f$ be a Boolean function on $n$ bits, $g$ on $m$ bits. 
	
	We begin with part (i) which is  a consequence of the standard fact that being pivotal is independent of the bit's value. Note that 
	\begin{equation}
		s_{f\circ g}\left(x^1_{1},\ldots,x^1_{m},\ldots,x^n_{1},\ldots,x^n_{m}\right) = \sum_{i=1}^n \mathbf{1}_{i\ \textrm{pivotal for}\ f} \cdot s_g\left(x^i_{1},\ldots,x^i_{m}\right).
	\end{equation}
	Taking expectations and using the fact that the event $\{i\ \textrm{pivotal for}\ f\}$ is measurable with respect to the bits $x_j^{i'}$ with $j \in [m]$ and $i' \in [n] \setminus \{i\}$, we get 
	\begin{align}
		s(f\circ g,\pi_p) &= \sum_{i=1}^n \mathbb E_{x \sim \pi_p} \left[ \mathbf{1}_{i\ \textrm{pivotal for}\ f} \cdot s_g\left(x^i_{1},\ldots,x^i_{m}\right)\right] \\
		&= \sum_{i=1}^n \mathbb P_{x \sim \pi_p} \left[ i\ \textrm{pivotal for}\ f\right] \cdot s(g,\pi_p) 
		= \underbrace{\mathbb E_{x \sim \pi_p}\left[ \sum_{i=1}^n \mathbf{1}_{i\ \textrm{pivotal for}\ f}\right]}_{= s(f,\pi_{g(p)})}
		\cdot s(g,\pi_p).
	\end{align}% 
	
	For the rest of the proof, let us write $x = (x_{B_1},\ldots,x_{B_n})$ with $x_{B_i} := (x^i_{1},\ldots,x^i_{m})$ and $y = y(x) := (g(x_{B_1}),\ldots,g(x_{B_n}))$.  To prove part (ii), we proceed in three steps.
	
	\noindent \textit{Step 1: Construction of a natural witness set $\mathcal I_f \circ \mathcal I_g \sim f\circ g$ from two witnesses sets $\mathcal I_f \sim f$ and $\mathcal I_g \sim g$}
	
	Given a witness set $\mathcal I_g$ for $g$ and $x \sim \pi_p$, we denote by $\mu_g$ the joint distribution of $(\mathcal I_g,x)$. Moreover, given a witness set $\mathcal I_f$ for $f$ and $x \sim \pi_{g(p)}$, we denote the conditional distribution of $\mathcal I_f$ given $x$ by $\mu_f^{x}$. We now construct $\mathcal I_f \circ \mathcal I_g$. Let $\mathcal I_g^1,\ldots,\mathcal I_g^n$ be random sets such that $\mathcal I_g^i$ is a witness set for $g$ on $x_{B_i}$ for every $i \in [n]$, and the pairs $(\mathcal I_g^i,x_{B_i})_{i=1}^n$ are i.i.d.\ with distribution $\mu_g$. Conditional on $y$, we now choose $\mathcal I_f^\star \subseteq [n]$ with conditional law $\mu_f^y$. Note that by construction, $\mathcal I_f^\star$ is conditionally independent of $x$ and $\mathcal I_g^1,\ldots,\mathcal I_g^n$ given $y$. We now define 
	\begin{equation}
		\mathcal I_f \circ  \mathcal I_g := \{(i-1)m + j : i \in \mathcal I_f^\star,\; j \in \mathcal I_g^i\} \;. 
	\end{equation}
	Note that $\mathcal I_f \circ  \mathcal I_g$ is a witness set for $f \circ g$.
	
	\noindent \textit{Step 2: If $\mathcal I_f$ is local for $x \sim \pi_{g(p)}$ and $\mathcal I_g$ is local for $x\in \pi_p$, then $\mathcal I_f \circ  \mathcal I_g$ is local for $x\sim \pi_p$.}	
	
	We need to show that for every $J \subseteq [nm]$, conditional on $\mathcal I_f \circ  \mathcal I_g = J$ and on $x\vert_J$, we have $x\vert_{J^\text{c}} \sim \pi_p$. To this end, we partition $J^\text{c}$ into $J_1$ and $J_2$ defined by 
	\begin{equation}
		J_1 := \{(i-1)m + j : i \not\in \mathcal I_f^\star,\; j \in [n]\}, \quad J_2 := \{(i-1)m + j : i \in \mathcal I_f^\star,\; j \not\in \mathcal I_g^i\}. 
	\end{equation}
	First, by the construction of $\mathcal I_f \circ  \mathcal I_g$ and the locality of $\mathcal I_f$, it follows that $x\vert_{J_1} \sim \pi_p$ given $\mathcal I_f \circ  \mathcal I_g = J$ and $x\vert_J$. Second, again by the construction of $\mathcal I_f \circ  \mathcal I_g$ and by the locality of $\mathcal I_g$, it follows that $x\vert_{J_2} \sim \pi_p$ given $\mathcal I_f \circ  \mathcal I_g = J$, $x\vert_J$, and $x\vert_{J_1}$.
	
	\noindent\textit{Step 3: $\E_{x \sim \pi_{p}}[\abs{\mathcal I_f \circ \mathcal I_g}] = \E_{x \sim \pi_{g(p)}}[\abs{\mathcal I_f}] \cdot \E_{x \sim \pi_{p}}[\abs{\mathcal I_g}]$.}
	
	It follows immediately from the construction of $\mathcal I_f \circ \mathcal I_g$ that 
	\begin{equation}\label{eq:proof-composition-1}
		\abs{\mathcal I_f \circ \mathcal I_g} =  \sum_{i=1}^n\mathbf{1}_{i \in \mathcal I_f^\star} \cdot \abs{\mathcal I_g^i}. 
	\end{equation}
	We first note that for every $i \in [n]$,
	\begin{align}
		&\mathbb E_{x \sim \pi_{p}} \left[\mathbf{1}_{i \in \mathcal I_f^\star} \cdot \abs{\mathcal I_g^i}\right] = \mathbb E_{x \sim \pi_{p}} \Big[ \underbrace{\mathbb E_{x \sim \pi_p} \left[\mathbf{1}_{i \in \mathcal I_f^\star} \mid x_{B_i}, \mathcal I_g^i \right]}_{=\mathbb P_{y \sim \pi_{g(p)}} \left[i \in \mathcal I_f^\star \mid y_i\right]} \cdot \abs{\mathcal I_g^i}\Big] = \mathbb P_{y \sim \pi_{g(p)}} \left[\mathbf{1}_{i \in \mathcal I_f^\star} \right] \cdot \mathbb E_{x \sim \pi_{p}} \left[\abs{\mathcal I_g^i}\right],
	\end{align}
	where in the last equality we used that $y_i$ is independent of $\mathbf{1}_{i \in \mathcal I_f^\star}$ by locality. Summing over $i\in [n]$ and using \eqref{eq:proof-composition-1}, we get
	\begin{equation}
		\E_{x \sim \pi_{p}}[\abs{\mathcal I_f \circ \mathcal I_g}] =
		\sum_{i=1}^n \mathbb P_{y \sim \pi_{g(p)}} \left[\mathbf{1}_{i \in \mathcal I_f^\star} \right] \cdot \mathbb E_{x \sim \pi_{p}} \left[\abs{\mathcal I_g^i}\right]
		= \E_{x \sim \pi_{g(p)}}[\abs{\mathcal I_f}] \cdot \E_{x \sim \pi_{p}}[\abs{\mathcal I_g}],
	\end{equation}
	yielding step 3.
	To conclude the proof of (ii), we take optimal local witness sets $\mathcal I_f \sim f$ for $x \sim \pi_{g(p)}$ and $\mathcal I_g \sim g$ for $x \sim \pi_p$, and deduce from step 3 that 
	\begin{equation}
		\ell(f \circ g,\pi_p) \le \E_{x \sim \pi_{p}}[\abs{\mathcal I_f \circ \mathcal I_g}] = \E_{x \sim \pi_{g(p)}}[\abs{\mathcal I_f}] \cdot \E_{x \sim \pi_{p}}[\abs{\mathcal I_g}]  = \ell(f,\pi_{g(p)}) \cdot \ell(g,\pi_{p}).
	\end{equation}	
	
	To prove part (iii), we follow a similar strategy in two steps.
	
	\noindent \textit{Step 1: Construction of a natural subcube partition $\mathcal C_f \circ \mathcal C_g \sim f\circ g$ from two subcube partitions $\mathcal C_f \sim f$ and $\mathcal C_g \sim g$}
	
	Let $C(y) \in \mathcal C_f$ be the subcube containing $y$. For every $x \in \{0,1\}^{n\cdot m}$, we now define $C(x) := \prod_{i=1}^n C_i(x)$, where
	\begin{equation}
		C_i(x) := \begin{cases}
			C(x_{B_i}) \in \mathcal C_g		&\textrm{if} \ i \ \textrm{is fixed by} \ C(y), \\
			\{0,1\}^m &\textrm{if} \ i \ \textrm{is not fixed by} \ C(y). \\			
		\end{cases}
	\end{equation}
	It now remains to show that $\{C(x)\}_{x \in \{0,1\}^{n\cdot m}}$ is a subcube partition for $f \circ g$, which we will denote by $\mathcal C_f \circ \mathcal C_g$. First, note that $C(y) \cap C(y') = \emptyset$ implies $C(x) \cap C(x') = \emptyset$. Hence, if $C(x) \cap C(x') \neq \emptyset$, then $C(y) = C(y')$. Hence, $C(x_{B_i}) = C(x'_{B_i})$ for every $i$ that is fixed by $C(y)$, and so $C(x) = C(x')$. Finally, we observe that the constructed subcube partition satisfies $\mathcal C_f \circ \mathcal C_g \sim f \circ g$.
	
	\noindent\textit{Step 2: The equality $sc(\mathcal C_f \circ \mathcal C_g,\pi_p) = sc(\mathcal C_f,\pi_{g(p)}) \cdot sc(\mathcal C_g,\pi_{p})$ holds, where we used the notation $sc(\mathcal C,\pi_p) = \mathbb E_{x \sim \pi_{p}} [c(\mathcal C,x)]$.}
	
	It follows immediately from the construction of $\mathcal C_f \circ \mathcal C_g$ that 
	\begin{equation}\label{eq:proof-composition-2}
		c(\mathcal C_f \circ \mathcal C_g,x) = \sum_{i=1}^n\mathbf{1}_{i \ \textrm{is fixed by} \ C(y)} \cdot c(\mathcal C_g,x_{B_i}) 
	\end{equation}
	We first note that for every $i \in [n]$,
	\begin{align}
		&\mathbb E_{x \sim \pi_{p}} \left[\mathbf{1}_{i \ \textrm{is fixed by} \ C(y)} \cdot c(\mathcal C_g,x_{B_i})\right] \\
		&= \mathbb E_{x \sim \pi_{p}} \Big[ \underbrace{\mathbb E_{x \sim \pi_p} \left[\mathbf{1}_{i \ \textrm{is fixed by} \ C(y)} \mid x_{B_i}\right]}_{=\mathbb P_{y \sim \pi_{g(p)}} \left[i \ \textrm{is fixed by} \ C(y) \mid y_i\right]} \cdot c(\mathcal C_g,x_{B_i})\Big] \\
		&= \mathbb P_{y \sim \pi_{g(p)}} \left[i \ \textrm{is fixed by} \ C(y) \right] \cdot sc(\mathcal C_g, \pi_p),
	\end{align}
	where in the last equality we used that $y_i$ is independent of $\{i \ \textrm{is not fixed by} \ C(y)\}$. Now, taking $\mathbb E_{x \sim \pi_p}$ in \eqref{eq:proof-composition-2}, we get
	\begin{equation}
		sc(\mathcal C_f \circ \mathcal C_g,\pi_p) = \sum_{i=1}^n  \mathbb P_{y \sim \pi_{g(p)}} \left[i \ \textrm{is fixed by} \ C(y) \right] \cdot sc(\mathcal C_g, \pi_p) =  sc(\mathcal C_f, \pi_{g(p)}) \cdot sc(\mathcal C_g, \pi_p),
	\end{equation}
	yielding step 2.
	To conclude the proof of (iii), we take optimal subcube partitions $\mathcal C_f \sim f$ and $\mathcal C_g \sim g$ and deduce from step 2 that 
	\begin{equation}
		sc(f \circ g,\pi_p) \le sc(\mathcal C_f \circ \mathcal C_g,\pi_p) =  sc(\mathcal C_f, \pi_{g(p)}) \cdot sc(\mathcal C_g, \pi_p) = sc(f,\pi_{g(p)}) \cdot sc(g,\pi_{p}).
	\end{equation}
	
	Finally, we show part (iv). We choose an optimal algorithm $A_f$ for $f$ and an optimal algorithm $A_g$ for $g$, and define an algorithm for $f \circ g$, denoted by $A_{f}\circ A_{g}$, as follows: Whenever the algorithm $A_f$ queries a bit $i\in [n]$, the algorithm $A_{f}\circ A_{g}$ determines the value of $y_i$ by using the algorithm $A_g$ on the bits corresponding to $(x_1^i,\ldots,x_m^i)$, and as soon as $y_i = g(x_{B_i})$ is determined, $A_{f}\circ A_{g}$ proceeds analogously with the next query of $A_f$. Clearly, this algorithm determines $f \circ g$ and from here, we conclude as in part (ii).
\end{proof}

\begin{remark}
	Analogously to part (i) of the last proof, one can show that for all Boolean functions $f$ and $g$,
	\begin{equation}
		b(f\circ g,\pi_p) \ge s(f,\pi_{g(p)}) \cdot b(g,\pi_p).
	\end{equation}
\end{remark}

\begin{remark} 
	Under the additional assumption that 
	\begin{equation}
		\mathbb E_{x \sim \pi_p} [ w_g(x) \mid g(x) = 1] = \mathbb E_{x \sim \pi_p} [ w_g(x) \mid g(x) = 0],
	\end{equation}
	one can prove that $w(f\circ g, \pi_p) \le w(f,\pi_{g(p)}) \cdot w(g,\pi_p)$. 
\end{remark}

\section{Examples}\label{sec:examples}

In the first subsection here, we consider classical examples of Boolean functions and study their distributional complexities. In the second subsection, we consider examples related to percolation theory.

\subsection{Classical examples} \label{subsec:classical-examples}

\begin{example}\label{ex:or-function}
	Concerning the AND-function on $n$ bits (see Example \ref{ex:tribes-function}), it is easy to see that some of its distributional complexities  are given by 
	\begin{align}
		&s(\textrm{AND}_n,\pi_p) = n \cdot p^{n-1}, \\
		&b(\textrm{AND}_n,\pi_p) = w(\textrm{AND}_n,\pi_p) = n \cdot p^n + (1-p^n), \\
		& a(\textrm{AND}_n,\pi_p) = \frac{1-p^n}{1-p}.
	\end{align}
	For the distributional subcube partition complexity, we have the following partial result.
	\begin{proposition}\label{prop:sc-equal-a-for-and-func}
		For $p \le 1/2$, we have $sc(\textrm{AND}_n,\pi_p)= a(\textrm{AND}_n,\pi_p)$.
	\end{proposition}
	\begin{proof}
		We prove by induction on $n$ that $sc(\textrm{AND}_n,\pi_p) = \frac{1-p^n}{1-p}$ with $n=1$ being obvious. Let $\mathcal C$ by a subcube partition $\mathcal C$ achieving the minimum $sc(\textrm{AND}_{n+1},\pi_p)$. Since $p \le 1/2$, we have $a(\textrm{AND}_n,\pi_p)<2$ for every $n\ge 1$, and hence $\mathcal C$ must contain a subcube with only 1 fixed bit. Without loss of generality, $\mathcal C$ contains the subcube $(0,\star,\ldots,\star)$. We are thus left with choosing an optimal subcube partition of $(1,\star,\ldots,\star)$, which is equivalent to choosing an optimal subcube partition for $\textrm{AND}_n$. Using the induction hypothesis,
		\begin{equation}
			\mathbb E_{x \sim \pi_p} \left[ c(\mathcal C,x)\right] = (1-p)\cdot 1 + p\cdot \left(\frac{1-p^n}{1-p}+1\right)  = \frac{1-p^{n+1}}{1-p}.
		\end{equation}%
	\end{proof}
\end{example}

\begin{example} \label{ex:iterated-3-MAJ} 
	Iterated 3-MAJORITY is the Boolean function $f^k$ on $n=3^k$ bits defined by $f := \textrm{MAJ}_3$ (see Example \ref{ex:majority-3}) and iteratively by $f^k := f \circ f^{k-1}$ for $k \ge 2$. 
	Fix $p=1/2$ and recall that the distributional complexities of $f = \textrm{MAJ}_3$ are given by
	\begin{align}
		&s(\textrm{MAJ}_3,\pi_{1/2}) = 3/2 ,\ b(\textrm{MAJ}_3,\pi_{1/2}) = 7/4, \ w(\textrm{MAJ}_3,\pi_{1/2})) = 2, \\ &\textrm{and} \ \ell(\textrm{MAJ}_3,\pi_{1/2}) = sc(\textrm{MAJ}_3,\pi_{1/2}) = a(\textrm{MAJ}_3, \pi_{1/2}) =5 /2. 
	\end{align}
	We note that $f = \textrm{MAJ}_3$ is balanced, i.e.\ $\mathbb P_{x \sim \pi_{1/2}}[f(x)=1] = 1/2$. Hence, by part (i) of Corollary \ref{cor:composition-distributional},
	\begin{equation}
		s(f^k,\pi_{1/2}) = (3/2)^k.
	\end{equation}
	Moreover, we note that $w_{f^k}(x)=2^k$ for all $x \in \{0,1\}^{3^k}$ and so
	\begin{equation}
		w(f^k,\pi_{1/2}) = 2^k.
	\end{equation}
	From the inequality of O'Donnell-Servedio (see Theorem \ref{thm:odonnell-servedio}) and part (iii) of Corollary \ref{cor:composition-distributional}, we get
	\begin{equation}
		(5/2)^k \ge a(f^k,\pi_{1/2}) \ge sc(f^k,\pi_{1/2}) \ge \ell(f^k,\pi_{1/2}) \overset{\textrm{(OS)}}{\ge} \left(s(f^k,\pi_{1/2})\right)^2 = (9/4)^k \ .
	\end{equation}		
\end{example}

\begin{example} \label{ex:majority-n}
	The MAJORITY function on $n$ bits (with $n$ odd) is defined by 
	\begin{equation}
		\textrm{MAJ}_n (x_1,\ldots,x_n) = \begin{cases} 	1 &\textrm{if} \sum_{i=1}^n x_i \ge \frac{n+1}{2}, \\
			0 &\textrm{if} \sum_{i=1}^n x_i \le \frac{n-1}{2}.
		\end{cases}
	\end{equation}
	For $p=1/2$, it is balanced and its distributional sensitivity is given by
	\begin{align}
		s(\textrm{MAJ}_n,\pi_{1/2}) = \frac{n+1}{2} \cdot \mathbb P_{x \sim \pi_{1/2}}\big[\sum_{i=1}^n x_i \in \left\{\tfrac{n-1}{2},\tfrac{n+1}{2}\right\}\big] \sim \sqrt{2/\pi} \cdot \sqrt{n}.
	\end{align}
	Next, we compute its distributional block sensitivity. Let $2\ell +1$ with $\ell \in \{0,\ldots, \tfrac{n-1}{2}\}$ be the difference between the number of 1's and 0's in $x$ with $\textrm{MAJ}_n(x)=1$. Then
	\begin{align}
		b_{\textrm{MAJ}_n}(x) = \lfloor\frac{\max\{\#\ \textrm{of}\ 1'\textrm{s in}\ x,\ \#\ \textrm{of}\ 0'\textrm{s in}\ x\}}{\ell +1}\rfloor =\lfloor\frac{n+2\ell +1}{2\ell +2}\rfloor
	\end{align}
	and so 
	\begin{align}
		b(\textrm{MAJ}_n,\pi_{1/2}) = 2 \sum_{\ell =0}^{\frac{n-1}{2}} \binom{n}{\frac{n+2\ell+1}{2}} \cdot 2^{-n} \cdot \lfloor\frac{n+2\ell +1}{2\ell +2}\rfloor \asymp \log(n)\cdot \sqrt{n}.
	\end{align}
	Since $w_{\textrm{MAJ}_n}(x)=\frac{n+1}{2}$ for every $x$, the distributional witness complexity is 
	\begin{align}
		w(\textrm{MAJ}_n,\pi_{1/2}) = \frac{n+1}{2}.
	\end{align}
	Finally, we note that the distributional algorithmic complexity is 
	\begin{align}
		a(\textrm{MAJ}_n,\pi_{1/2}) = n - o(n).
	\end{align}	
\end{example}

\begin{example}\label{ex:balanced-tribes-function}
	TRIBES with $\ell$ tribes of size $m$ is the Boolean function on $n = \ell \cdot m$ bits defined by 
	\begin{equation}
		\textrm{TRIBES}_{\ell,m} = \textrm{OR}_\ell \circ \textrm{AND}_m  .
	\end{equation}
	We fix  $p=1/2$ and $\ell =2^m$, and note that for these choices 
	\begin{equation}
		\mathbb P_{x \sim \pi_{1/2}}\big[\textrm{TRIBES}_{2^m,m}(x)=1\big] = 1 - (1-\tfrac{1}{2^{m}})^{2^m} \to 1- e^{-1} \ \text{as}\ m\to \infty.
	\end{equation}
	It is easy to see that the distributional complexities are given by
	\begin{align}
		&s\big(\textrm{TRIBES}_{2^m,m},\pi_{1/2}\big) = m\cdot 2^m \cdot \frac{1}{2^{m-1}} \big(1-\tfrac{1}{2^m}\big)^{2^m-1} \sim 2m \cdot e^{-1} \sim \frac{2}{e} \log_2(n), \\
		&b\big(\textrm{TRIBES}_{2^m,m},\pi_{1/2}\big) = w\big(\textrm{TRIBES}_{2^m,m},\pi_{1/2}\big) \sim 2^m\cdot e^{-1}  + m \cdot (1-e^{-1}) \asymp \frac{n}{\log_2(n)}, \\
		&a\big(\textrm{TRIBES}_{2^m,m},\pi_{1/2}\big)  \asymp \frac{n}{\log_2(n)}.
	\end{align}
	Thus, we also have $sc(\textrm{TRIBES}_{2^m,m},\pi_{1/2})  \asymp \frac{n}{\log_2(n)}$ and $\ell(\textrm{TRIBES}_{2^m,m},\pi_{1/2})  \asymp \frac{n}{\log_2(n)}$.
\end{example}

\subsection{Percolation functions} \label{subsec:percolation-examples}

Let $G=(V,E)$ be a finite multigraph. A percolation configuration $\omega \in \{0,1\}^E$ declares edges \emph{open} if $\omega(e)=1$ or \emph{closed} if $\omega(e)=0$. We identify $\omega$ with the subgraph of open edges $G_\omega := (V,E_\omega)$, where   $E_\omega := \{ e \in E: \omega(e)=1\}$, and introduce some notation to describe its connectivity properties. A \emph{path} $\gamma = (\gamma_i)_{i=0}^n \subseteq V$ of  length $n$ is a sequence of distinct vertices with $\{\gamma_{i-1},\gamma_i\} \in E$ for all $1 \le i \le n$, and the path is called \emph{open} if $\omega(\{\gamma_{i-1},\gamma_i\})=1$ for  all $1 \le i \le n$. We say that two subsets $A, B \subseteq V$ are \emph{connected} if there exists an open path that starts in $A$ and ends
in $B$, and denote this event by $\{A \leftrightarrow B\}$. 

\begin{definition}
	A Boolean function $f:\{0,1\}^n \to \{0,1\}$ is called a \emph{percolation function} if it is of the form 
	\begin{equation}
		f(\omega) = \begin{cases}
			1 &\textrm{if}\ \omega \in \{A \leftrightarrow B\}, \\
			0 &\textrm{otherwise}, \\
		\end{cases}
	\end{equation}
	for some multigraph $G=(V,E)$ with $n$ edges and some $A,B \subseteq V$.
\end{definition}

\begin{remark}
	We note that the function $\textrm{AND}_2 \circ \textrm{OR}_2$ is a percolation function according to the previous definition, but it cannot be represented by a simple graph.
\end{remark}

AND-functions, OR-functions and TRIBES-functions are percolation functions. However, MAJORITY-functions are not percolation functions. We recall from Example \ref{ex:or-function} that for every $p \in (0,1)$ and $n \ge 1$,
\begin{equation}
	a(\textrm{AND}_n,\pi_p) > w(\textrm{AND}_n,\pi_p) = b(\textrm{AND}_n,\pi_p) > s(\textrm{AND}_n,\pi_p).
\end{equation}
In fact, Corollary \ref{cor:percolation-functions} will show that the above comparisons are general properties of percolation functions. 

\begin{proposition}\label{prop:witness-block-equality}
	For every percolation function $f$ and for every $\omega \in \{0,1\}^n$,
	\begin{equation}
		w_f(\omega) = b_f(\omega).
	\end{equation}%
	In particular, for every $p \in [0,1]$, we have $w(f,\pi_{p}) = b(f,\pi_{p})$.
\end{proposition}

\begin{proof}
	Let $\omega$ be such that $f(\omega)=1$. Since $f$ is monotone, it suffices to consider witnesses $W \subseteq E_\omega$ and blocks that are subsets of $ E_\omega$. The witness sets of minimum size are now the shortest open paths from $A$ to $B$ and so $w_f(\omega)$ is the length of a shortest path from $A$ to $B$ in the multigraph $G_\omega$. The maximum number $b_f(\omega)$ of disjoint blocks that flip the outcome of $f$ is thus equal to the maximum number of disjoint $(A,B)$-edge-cutsets for $G_\omega$. While the inequality $b_f(\omega)\le w_f(\omega)$ holds for every Boolean function $f$ and every $\omega$, the inequality $b_f(\omega)\ge w_f(\omega)$ follows by considering the disjoint $(A,B)$-edge-cutsets for $G_\omega$ defined by 
	\begin{equation}
		\left\{e \in E_\omega : d_{G_\omega}(e,A) = i \right\}
	\end{equation}
	for $0 \le i \le w_f(\omega)-1$.
	
	Now, let $\omega$ be such that $f(\omega)=0$. Since $f$ is monotone, it suffices to consider witnesses $W \subseteq E\setminus E_\omega$ and blocks that are subsets of $E\setminus E_\omega$. By collapsing each connected component of $G_\omega$ into a single vertex, we define a new multigraph $\hat{G}=(\hat{V},\hat{E})$ whose vertices $\hat{V}$ correspond to the connected components of $G_\omega$ and whose edges $\hat{E}$ correspond to the closed edges $E\setminus E_\omega$ in the original multigraph $G$. $w_f(\omega)$ is now equal to the size of a minimum $(A,B)$-edge-cutset for $\hat{G}$ and $b_f(\omega)$ is equal to the maximal number of edge-disjoint paths from $A$ to $B$ in $\hat{G}$. The equality $w_f(\omega)=b_f(\omega)$ follows from Menger's Theorem.
\end{proof}

The following corollary is a direct consequence of Propositions \ref{prop:s-b-equal}, \ref{prop:sc-w-equal}, and \ref{prop:witness-block-equality}.
\begin{corollary}\label{cor:percolation-functions}
	For every percolation function $f$ which depends on at least 2 bits, we have that 
	\begin{equation}
		sc(f,\pi_p) > w(f,\pi_p) = b(f,\pi_p) > s(f,\pi_p).
	\end{equation}
\end{corollary}

We are not aware of any monotone Boolean function $f$ satisfying $a(f,\pi_p) > sc(f,\pi_p)$ (see Question \ref{question:monotone-subcube-algorithmic}) or $sc(f,\pi_p) > \ell(f,\pi_p)$, and thus also not of a percolation function satisfying this.

\begin{example} \label{ex:square-crossing}
	Consider the square lattice $\Z^2$ and for $m\ge 1$, let $G_m=(V_m,E_m)$ be the graph containing the vertices and edges in the square $[0,m+1]\times[0,m]$. We define the square-crossing function $f_m: \{0,1\}^{E_m} \to \{0,1\}$ by
	\begin{equation}
		f_m(\omega) = 1  \iff   \omega_n \in \begin{minipage}[c]{.1\textwidth}
			\includegraphics[width=\textwidth]{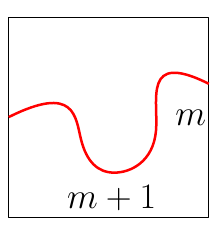}
		\end{minipage},
	\end{equation}
	where the drawing represents a connection from the left side to the right side. This corresponds to choosing $A=\{(0,0),\ldots,(0,m)\}$ and $B=\{(m+1,0),\ldots,(m+1,m)\}$. We point out that $\mathbb P_{\omega \sim \pi_{1/2}}[f_m = 1] = 1/2$ for all $m$ by symmetry and planar duality, and that $\mathbb P_{\omega \sim \pi_{p}}[f_m = 1]$ is known to go to 0 (resp.\ 1) as $m \to \infty$ for $p<1/2$ (resp.\ $p>1/2)$. 
	\begin{proposition}\label{prop:off-critical-percolation}
		Fix $p \in (0,1)$ such that $p\neq 1/2$. 
		\begin{itemize}
			\item[(i)] There exist constants $c, c' >0$ depending only upon $p$ such that for every $m\ge 1$,
			\begin{equation}
				e^{-c m} \le s(f_m,\pi_p) \le e^{-c' m}.
			\end{equation}
			\item[(ii)] There exists a constant $C< \infty$ depending only upon $p$ such that for every $m\ge 1$,
			\begin{equation}
				m+1 \le b(f_m,\pi_{p}) = w(f_m,\pi_{p}) \le \ell(f_m,\pi_p)\le  sc(f_m,\pi_p) \le a(f_m,\pi_p) \le C m.
			\end{equation}
		\end{itemize}	
	\end{proposition}
	\begin{proof}
		By duality, it suffices to consider $p\in (0,1/2)$. For the lower bound in (i), consider the event that all edges of the form $\{(x,0),(x+1,0)\}$ with $0 \le x \le m$ are open and  all edges of the form $\{(0,y),(1,y)\}$ with $1 \le y \le m$ are closed. Note that on this event, the edge $\{(0,0),(1,0)\}$ is guaranteed to be pivotal and that this event has probability $p^{m+1}(1-p)^m $.
		For the upper bound in (i), note that for any edge $e \in E_m$, $e$ being pivotal essentially implies the existence of a left-right crossing. The probability of the latter is exponentially small in $m$ (see \cite{Grimmett1999}). Since we only have order $m^2$ edges, the upper bound follows.
		
		For the first inequality in $(ii)$, we actually show that $b_{f_m}(\omega) \ge m+1$ for every $\omega$. To see this, if $f_m(\omega)=1$, we obtain $m+1$ disjoint blocks 
		\begin{equation}
			B_i := \left\{\{(i,y),(i+1,y)\} \in E_m : \omega(\{(i,y),(i+1,y)\}) = 1 \right\}, \quad 0 \le i \le m,
		\end{equation}
		whose flipping changes the value of $f_m$. Similarly,  if $f_m(\omega)=0$, we obtain $m+1$ disjoint blocks 
		\begin{equation}
			B_i := \left\{\{(x,i),(x+1,i)\} \in E_m : \omega(\{(x,i),(x+1,i)\}) = 0 \right\}, \quad 0 \le i \le m,
		\end{equation}
		whose flipping changes the value of $f_m$. The equality  $b(f_m,\pi_{p}) = w(f_m,\pi_{p})$ follows from Proposition \ref{prop:witness-block-equality}. All the other inequalities follow from the hierarchy (see Proposition \ref{prop:hierarchy-dist}), except the last which we now argue by describing a suitable algorithm: Fixing an arbitrary ordering of the $m+1$ vertices on the left side of the square, the algorithm explores the cluster of each vertex one-by-one, thereby checking if this vertex is connected to the right side of the square. Since the expected size of the cluster of any vertex is of order 1 (see \cite{Grimmett1999}), the result follows. 
	\end{proof}
	
	We now look at the distributional complexities at the critical point $p_c=1/2$. 
	\begin{proposition}\label{prop:critical-percolation}
		Let $p = 1/2$. There exists $c >0$ such that
		\begin{itemize}
			\item[(i)]$	m^c \lesssim s(f_m,\pi_{1/2}) \lesssim  m^{1-c}$,
			\item[(ii)]$m^{1+c} \lesssim b(f_m,\pi_{1/2}) = w(f_m,\pi_{1/2}) \le \ell(f_m,\pi_p) \le sc(f_m,\pi_{1/2}) \le a(f_m,\pi_{1/2}) \lesssim m^{2-c}$.
		\end{itemize}	
	\end{proposition}
	
	\begin{proof}
		For (i), we refer to Lemma 6.5 in \cite{Garban2014} for the lower bound and to  Proposition 6.6 in \cite{Garban2014} for the upper bound. 
		
		For the upper bound in (ii), we make use of a standard algorithm that determines $f_m$ by exploring the cluster of each of the points on the left side. We note that an edge $e \in E_m$ is queried by this algorithm before determining $f_m$ only if there is an open path in $G_m$ from one of the endpoints of $e$ to the left side of the square. Therefore,	
		\begin{align}
			a(f_m,\pi_{1/2}) &\le \sum_{e \in E_m} \P_{x \sim \pi_{1/2}}\big[e \xleftrightarrow{\; G_m\; } \{0\}\times[0,m]\big] \\
			& \le (2m+1) \cdot \sum_{k=0}^{m}  \P_{x \sim \pi_{1/2}}\left[(0,0) \xleftrightarrow{\;\Z^2\;} \partial\left([k,k]^2\right) \right],
		\end{align}
		where the k-th summand accounts for the edges with minimal $x$-coordinate $k$.
		By Theorem 11.89 in \cite{Grimmett1999}, there are constants $c,C \in (0,\infty)$ such that for every $k\ge 1$, 
		\begin{equation}
			\P_{x \sim \pi_{1/2}}\left[(0,0) \xleftrightarrow{\; \Z^2\; } \partial\left([-k,k]^2\right) \right] \le C k^{-c}. 
		\end{equation}
		From this, the upper bound in (ii) follows.
		
		The equality  $b(f_m,\pi_{1/2}) = w(f_m,\pi_{1/2})$ in (ii) follows from Proposition \ref{prop:witness-block-equality}, and all the other inequalities,  except the first, follow from the hierarchy (see Proposition \ref{prop:hierarchy-dist}). Finally, the fact that $w(f_m,\pi_{1/2})] \gtrsim m^{1+c}$ is stated in (1.1) in \cite{Damron2017}.
	\end{proof}
	
	Critical percolation on the square lattice is conjectured to be conformally invariant, which would yield exact values for the so-called arm exponents. Using these conjectured values, we arrive at the following conjecture. We note that in the case of critical site percolation on the triangular lattice, conformal invariance is known, and hence so are the arm exponents \cite{Smirnov2001}. 
	\begin{conjecture} \label{conjecture-critical-percolation} ~
		\begin{itemize}
			\item[(i)] $s(f_m,\pi_{1/2}) = m^{3/4 + o(1)}$,
			\item[(ii)] $w(f_m,\pi_{1/2})  \le m^{4/3 + o(1)}$
			\item[(iii)] $\ell(f_m,\pi_{1/2}) \ge m^{3/2 + o(1)}$,
			\item[(iv)] $a(f_m,\pi_{1/2}) \le  m^{7/4 + o(1)}$.
		\end{itemize}	
	\end{conjecture}
	We now comment on this conjecture: For (i), the probability of being pivotal is essentially $m^{-\alpha_4}$ where the four-arm exponent $\alpha_4$ is conjectured to be $5/4$. Using the inequality of O'Donnell-Servedio (see Theorem \ref{thm:odonnell-servedio}), (iii) follows immediately from (i). For (ii), one first notes that $w_{f_m}(\omega)$ is at most the length of the lowest crossing (when $f_m(\omega) =1$). Since an edge being on the lowest crossing entails a three-arm event starting at that edge and since the three-arm exponent is conjectured to be $2/3$, the result would follow. For (iv), one makes use of an algorithm that ``follows the interface''. Being queried then corresponds to the two-arm exponent which is conjectured to be $1/4$ (see \cite[Theorem~8.4]{Garban2014}).
	
	Non-rigorous arguments suggest that $w(f_m,\pi_{1/2})$ behaves like $m^{1.130\cdots}$ and $a(f_m,\pi_{1/2})$ behaves like $m^\alpha$ with $\alpha \in [1.5,1.6]$ (see \cite{Peres2007}). Finding the exact exponents in these two cases are well-known open problems.

\end{example}

\section{Asymptotic separations}\label{sec:asymp-separation}

In Section \ref{sec:defs-hierarchy}, we have seen examples showing that the five complexity measures are distinct, both in the deterministic and distributional case. In this section, we study asymptotic separations.

\subsection{Review of the deterministic case}\label{subsec:asymp-separation-det}

Asymptotic separations are well-studied in the deterministic case and we refer the reader to the recent paper \cite{Aaronson2021} (see, in particular, Table 1) for an excellent overview. We summarize a small fraction of the known results in the following theorem.
\begin{theorem}\label{thm:asymp-sep-det}
	There exist sequences of Boolean functions $(f^{(i)}_n)_{n \ge 1}$, $1 \le i\le 4,$ such that asymptotically as $n \to \infty$, 
	\begin{equation}
		\frac{a_D(f^{(1)}_n)}{sc_D(f^{(1)}_n)} \to \infty,\quad  \frac{sc_D(f^{(2)}_n)}{w_D(f^{(2)}_n)} \to \infty,\quad  \frac{w_D(f^{(3)}_n)}{b_D(f^{(3)}_n)} \to \infty, \ \textrm{and}\quad  \frac{b_D(f^{(4)}_n)}{s_D(f^{(4)}_n)} \to \infty.
	\end{equation}
\end{theorem}

We refer to \cite{Ambainis2015,Kothari2015} for the first, to \cite{Gilmer2013} for the third, and to \cite{Rubinstein1995} for the fourth separation. For the second, the TRIBES function (see Example \ref{ex:tribes-function}) with $\ell$ tribes of size $m$ provides such an example assuming $\min\{\ell,m\} \to \infty$ since $w_D(\textrm{TRIBES}_{\ell,m})=\max\{\ell,m\}$ and $sc_D(\textrm{TRIBES}_{\ell,m})= n = \ell \cdot m$.

\subsection{Asymptotic separation for distributional complexities}\label{subsec:asymp-separation-dist}

The following theorem yields a first general asymptotic separation \cite{ODonnell2007}. 

\begin{theorem}[O'Donnell-Servedio]\label{thm:odonnell-servedio}
	For every monotone Boolean function $f$ and for every $p \in [0,1]$,
	\begin{equation}
		sc(f,\pi_{p}) \ge 4p(1-p) \cdot \big(s(f,\pi_{p})\big)^2\ .
	\end{equation}
\end{theorem}
In fact, the same inequality holds with the left-hand side replaced by $\ell(f,\pi_{p})$. The easiest way to see this is by adapting the proof given in \cite[Thm.~8.8]{Garban2014}.

Based on the examples studied in Subsection \ref{subsec:classical-examples}, we obtain the following result. 

\begin{theorem}\label{thm:asymp-sep-dist}
	There exist sequences of Boolean functions $(f^{(i)}_n)_{n \ge 1}$, $1 \le i\le 3,$ such that asymptotically as $n \to \infty$, 
	\begin{equation}
		\frac{\ell(f^{(1)}_n,\pi_{1/2})}{w(f^{(1)}_n,\pi_{1/2})} \to \infty,\quad  \frac{w(f^{(2)}_n,\pi_{1/2})}{b(f^{(2)}_n,\pi_{1/2})} \to \infty, \ \textrm{and}\quad  \frac{b(f^{(3)}_n,\pi_{1/2})}{s(f^{(3)}_n,\pi_{1/2})} \to \infty.
	\end{equation}
\end{theorem}

The reader will immediately notice that asymptotic separations between distributional algorithmic complexity, subcube partition complexity and local witness complexity are missing. Before discussing the difficulties in asymptotically separating distributional algorithmic complexity and subcube partition complexity, we prove the three asymptotic separations. 

\begin{proof}
	\begin{enumerate}
		\item Consider iterated 3-MAJORITY $f^k$ on $n=3^k$ bits as defined in Example \ref{ex:iterated-3-MAJ}. Then we have
		\begin{equation}
			\frac{\ell(f^k,\pi_{1/2})}{w(f^k,\pi_{1/2})}  \ge (9/8)^k = n^{\log_3(9/8)} \to \infty \ \textrm{as} \ n \to \infty.
		\end{equation}
		\item Consider $\textrm{MAJ}_n$ on $n$ bits (with $n$ odd) as defined in Example \ref{ex:majority-n}. Then we have
		\begin{equation}
			\frac{w(\textrm{MAJ}_n,\pi_{1/2})}{b(\textrm{MAJ}_n,\pi_{1/2})} \asymp \frac{\sqrt{n}}{\log(n)} \to \infty, \ \textrm{as} \ n \to \infty.
		\end{equation}
		\item 
		Consider  $\textrm{TRIBES}_{2^m,m}$  with $2^m$ tribes of size $m$ as studied in Example \ref{ex:balanced-tribes-function}. Then we have 
		\begin{equation}
			\frac{b\left(\textrm{TRIBES}_{2^m,m},\pi_{1/2}\right)}{s\left(\textrm{TRIBES}_{2^m,m},\pi_{1/2}\right)} \asymp \frac{n}{\log(n)^2} \to \infty, \ \textrm{as} \ n \to \infty.
		\end{equation}
	\end{enumerate}
\end{proof}

We end this section with a proposed example for asymptotically separating distributional algorithmic complexity and distributional subcube partition complexity. Let $f$ be the ALL-EQUAL function on $n=3$ bits (see Example \ref{ex:constant-function}).
We want to iterate $f$ but since $f$ is not balanced, it makes things simpler to XOR $f$ with one bit so that it is balanced and hence all its iterates are balanced. In fact, we believe that in order for this example to have a chance of working, we need to take $f$ and XOR it with a large number of bits. Denote by $f\oplus \textrm{PAR}_k$ the function on $3+k$ bits which is $f$ XORed with $k$ bits. Clearly, when $p=1/2$, the  distributional algorithmic complexity and
distributional subcube partition complexity are $5/2+k$ and $9/4+k$ respectively. We now consider the sequence of Boolean functions $(f \oplus \textrm{PAR}_k)^\ell$ with $k$ fixed and indexed by $\ell$.  By Corollary \ref{cor:composition-distributional}, we know that the distributional subcube partition complexity of $(f\oplus \textrm{PAR}_k)^\ell$  is at most  $(9/4+k)^\ell$. The hope is that for $k$ large, the distributional algorithmic complexity of $(f\oplus \textrm{PAR}_k)^\ell$  is precisely $(5/2+k)^\ell$ for every $\ell \ge 1$. If this were the case, we would have
the desired separation. The above would correspond to the optimal algorithm for $(f\oplus \textrm{PAR}_k)^\ell$ just being the optimal algorithm for $(f\oplus \textrm{PAR}_k)^\ell$ iterated $\ell$ times. The reason this might be believable is that whenever one deviates from the latter, one has an initial cost of $k$ corresponding to looking at the $k$ added XOR bits which one always needs to do first. If $k$ is very large, this cost is too much and it is not worth deviating from the simple "iterative" algorithm.

This approach will naturally lead to Section \ref{sec:alg-comp-partial-info} where we introduce the notion of partial information revealment. The connection arises by viewing the querying of the $k$ parity bits of $f\oplus \textrm{PAR}_k$ as providing partial information on the value of $f\oplus \textrm{PAR}_k$. We will revisit this approach in Section \ref{sec:alg-comp-composition}.

\section{Polynomial relations}\label{sec:polynomial-relations}

In 2019, Huang proved the long-standing sensitivity conjecture \cite{Huang2019}, which, given previous results, implies that $a_D(f) \le s_D(f)^C$ for every Boolean function $f$. In other words, all deterministic complexity measures are polynomially related to each other. We first review the polynomial relations in the deterministic case and then move to the distributional setting.

\subsection{Review of the deterministic case}\label{subsec:polynomial-relations-det}

The five deterministic complexity measures discussed so far are actually part of a larger family studied in computational complexity theory. Let us introduce one additional complexity measure which was particularly important in the proof of the sensitivity conjecture. We denote by $\text{deg}(f)$ the \emph{degree} of the multilinear polynomial $p$ that represents $f$, i.e.\ $p(x) = f(x)$ for all $x \in \{0,1\}^n$. 

\begin{theorem}\label{thm:poly-rel-det}
	For every Boolean function $f$,
	\begin{equation}
		a_D(f) \overset{\text{(1)}}{\le} w_D(f) \cdot b_D(f), \ w_D(f) \overset{\text{(2)}}{\le} b_D(f) \cdot s_D(f), \ b_D(f) \overset{\text{(3)}}{\le} \text{deg}(f)^2 \ \textrm{and} \ \text{deg}(f)  \overset{\text{(4)}}{\le} s_D(f)^2.
	\end{equation}
	In particular, all these deterministic complexity measures are polynomially related.
\end{theorem}
Inequality (1) was proven in \cite{Beals2001}, (2) in \cite{Nisan1989}, (3) in \cite{Nisan1989} with an extra factor $2$ and then improved to the above in \cite{Tal2013}, and (4) in \cite{Huang2019}. We refer to the survey \cite{Buhrman2002} and to  \cite{Aaronson2021} for more background.

Moreover, Nisan showed equality of witness complexity and sensitivity for monotone Boolean functions \cite[Proposition 2.2]{Nisan1989}.
\begin{proposition}\label{prop:monotone-fct-det}
	For every monotone Boolean function $f$,
	\begin{equation}
		w_D(f) = b_D(f)=s_D(f).
	\end{equation}%
\end{proposition}

\subsection{Polynomial relations for distributional complexities}\label{subsec:polynomial-relations-dist}

We begin by considering the AND-function on $n$ bits from  Example \ref{ex:or-function} whose distributional sensitivity and block sensitivity are given by 
\begin{equation}
	s(\textrm{AND}_n,\pi_p) = n \cdot p^{n-1}, \quad \text{and}\quad b(\textrm{AND}_n,\pi_p) = 1 + (n-1) \cdot p^{n} \ge 1.
\end{equation}
Hence, for any $p \in (0,1)$, we do not have a polynomial relation between $s(\textrm{AND}_n,\pi_p)$ and $b(\textrm{AND}_n,\pi_p)$. However, note that for any $p \in (0,1)$, $\textrm{Var}_p(\textrm{AND}_n) \to 0$ as $n \to \infty$.

Our goal is thus to obtain polynomial relations for Boolean functions with variances bounded away from 0.
\begin{question}
	Fix $c>0$. Are there any polynomial relations between
	\begin{equation}
		a(f,\pi_p), \quad 	sc(f,\pi_p), \quad \ell(f,\pi_p), \quad 	w(f,\pi_p), \quad 	b(f,\pi_p), \quad \text{and} \quad	s(f,\pi_p)
	\end{equation}
	for general Boolean functions $f$ satisfying $\textrm{Var}_p(f) \ge c$? 
\end{question}
We are also interested in such relations for restricted classes of Boolean functions, for example when restricting to monotone, transitive, or symmetric Boolean functions. 

A first negative answer to this question is given by the function $\textrm{TRIBES}_{2^m,m}$ with $2^m$ tribes of size $m$ (see Example \ref{ex:balanced-tribes-function}). It is transitive, monotone, and for $p=1/2$, 
\begin{equation}
	\textrm{Var}_{1/2}\big(\textrm{TRIBES}_{2^m,m}\big)  \to e^{-1}(1- e^{-1}) \ \text{as}\ m\to \infty.
\end{equation}
However, its distributional sensitivity and block sensitivity are \emph{not} polynomially related since 
\begin{align}
	s\big(\textrm{TRIBES}_{2^m,m},\pi_{1/2}\big)\asymp \log_2(n), \quad \textrm{but}
	\quad b\big(\textrm{TRIBES}_{2^m,m},\pi_{1/2}\big) \asymp \frac{n}{\log_2(n)}.
\end{align}

A first positive answer under certain conditions follows from the so-called OSSS inequality \cite{ODonnell2005}.

\begin{theorem}[O'Donnell-Saks-Schramm-Servedio]
	Let $f$ be a Boolean function on $n$ bits. For every algorithm $A$ and for every $p\in (0,1)$,
	\begin{equation}
		\textrm{Var}_{p}(f) \le 4p(1-p) \sum_{i=1}^n \mathbb P_{x \sim \pi_p}[i \in A_T(f)] \cdot \mathbb P_{x \sim \pi_p}[i\ \text{pivotal for}\ f],
	\end{equation}
	where $A_T(f)$ denotes the (random) set of bits queried by $A$ before determining $f$. \\
	In particular, for every transitive $f$,
	\begin{equation}
		n \cdot \textrm{Var}_{p}(f) \le 4p(1-p) \cdot a(f,\pi_p) \cdot s(f,\pi_p).
	\end{equation}
\end{theorem}

We deduce the following immediate corollary.

\begin{corollary}\label{cor:poly-rel-dist}
	Let $p \in (0,1)$ and let $(f_n)_{n\ge 1}$ be a sequence of transitive Boolean functions (with $f_n$ on $n$ bits) satisfying for some $c,\epsilon >0$ and $C<\infty$,
	\begin{equation}
		\textrm{Var}_{p}(f_n) \ge c \quad \text{and}\quad  a(f,\pi_p) \le  C\cdot  n^{1-\epsilon}, \quad \forall n\ge 1.
	\end{equation}
	Then,
	\begin{equation}
		s(f_n,\pi_p) \ge  \frac{c}{C^\frac{1}{1-\epsilon}} \cdot  \big(a(f_n,\pi_p)\big)^{\frac{\epsilon}{1-\epsilon}}
	\end{equation}
	and hence $s(f_n,\pi_p)$ and $a(f_n,\pi_p)$ are polynomially related.
\end{corollary}

\begin{remark}
	In the case of critical planar percolation, this corollary can be used to show polynomial relations at criticality (see Example \ref{ex:square-crossing}).
\end{remark}

We close this section with a geometric perspective  on the ratio of $sc(f,\pi_{1/2})$ and $s(f,\pi_{1/2})$, namely
\begin{equation}
	\frac{sc(f,\pi_{1/2})}{s(f,\pi_{1/2})} = \frac{\inf_{\mathcal C : E(\mathcal C) \subseteq E(f)} \abs{\Delta_E(\mathcal C)}}{\abs{\Delta_E(f)}} =: \gamma(f),
\end{equation} 
where $E_n$ denotes the edges of the cube $\{0,1\}^n$, $E(f) := \{\{x,y\} \in E_n : f(x)=f(y)\}$, $E(\mathcal C) := \{\{x,y\} \in E_n : C(x)=C(y)\}$, and the two edge boundaries are defined as $\Delta_E(f) = E_n \setminus E(f)$ and $\Delta_E(\mathcal C) = E_n \setminus E(\mathcal C)$. The equality above is relatively straightforward to check and Theorem \ref{thm:odonnell-servedio} immediately yields $\gamma(f) \ge s(f,\pi_{1/2})$. Establishing upper bounds on $\gamma(f)$ in terms of other distributional complexity measures could then be a way to obtain polynomial relations.

\section{Distributional algorithmic complexity with partial information}\label{sec:alg-comp-partial-info}

In this section, we will only be dealing with the product measure $\pi_{1/2}$. A standard algorithm queries the bits one by one and pays a cost of 1 for each query. 
Our goal is to allow algorithms to ask for partial information concerning the value of a bit at a cost depending on the amount of information asked for. 

\subsection{Distributional $\Gamma$-algorithmic complexity}

Consider $x=(x_1,\ldots,x_n)$, where each $$x_i = (x_i(t))_{t=0}^{T_{\{0,1\}}}$$ is an independent Brownian motion started at $1/2$ and stopped when hitting the set $\{0,1\}$. 
A \emph{generalized algorithm} $A$ starts in the state $(1/2,\ldots,1/2)$ and initially picks some $i \in [n]$ and $p \in (1/2,1]$. To obtain the new state, we let $x_i$ evolve up to time $T_{\{1-p,p\}}$, the hitting time of $\{1-p,p\}$. The new state is then
\begin{equation}
	\left(1/2,\ldots,1/2,x_i(T_{\{1-p,p\}}),1/2,\ldots,1/2\right).
\end{equation}
At any later stage, if the present state is $(p_1,\ldots,p_n)$, $A$ chooses some $j \in [n]$ and $p \in (\max\{1-p_j,p_j\},1]$, and goes to state
$(p_1,\ldots,p_{j-1},x_j(T_{\{1-p,p\}}),p_{j+1},\ldots,p_n)$. Note that $p_j \to 1- p$ with probability $\frac{p - p_j}{p-(1-p)}$ and  $p_j \to p$ with probability $\frac{p_j - (1-p)}{p-(1-p)}$.
The algorithm continues until reaching a state $(p_1,\ldots,p_n) \in \{0,1\}^n$. We will only consider algorithms that almost surely terminate after finitely many steps. 
Note that standard algorithms as defined in Section \ref{sec:defs-hierarchy} correspond to those algorithms which always pick $p=1$.

A \emph{cost function} $\Gamma: [1/2,1] \to [0,1]$ is a non-decreasing function satisfying $\Gamma(1/2)=0$ and $\Gamma(1)=1$.	
Given a Boolean function $f$ on $n$ bits and a cost function $\Gamma$, we define the \emph{$\Gamma$-cost of $A$ for determining $f$} as 
\begin{equation}
	c_\Gamma(A,f) = \E_{x}\left[\sum_{i=1}^n \Gamma(\max\{1-\bar{p}_i,\bar{p}_i\}) \right],
\end{equation} 
where $\bar{p}_i$ is the state of  bit $i$ at the first moment at which the $p_j$'s that are $0$ or $1$ determine $f$.
The motivation for this cost function is to assign a cost 
$\Gamma(p) - \Gamma(\max\{1-p_i,p_i\})$
for the transition from $p_i$ or $1-p_i$ to $\{p,1-p\}$.
Finally, we define the \emph{distributional $\Gamma$-algorithmic complexity} as
\begin{equation}\label{eq:1}
	a_\Gamma(f,\pi_{1/2}) = \inf_A c_\Gamma(A,f).
\end{equation}
We note that $w(f,\pi_{1/2}) \le 	a_\Gamma(f,\pi_{1/2}) \le 	a(f,\pi_{1/2})$ for all $\Gamma$ and for all $f$. Note that for the trivial choice $\Gamma = \mathbf{1}_{(1/2,1]}$ (respectively $\Gamma = \mathbf{1}_{\{1\}}$), the right (respectively left) inequality saturates. 
We recall that for PARITY functions and ADDRESS functions (see Example \ref{ex:address}),  $w(f,\pi_{1/2}) = a(f,\pi_{1/2})$.
The next example illustrates that this new complexity measure can have more interesting behavior. 

\begin{example}\label{ex:gamma-algo-AND-2}
	The AND function on $n=2$ bits (see Examples \ref{ex:tribes-function} and \ref{ex:or-function}) has distributional complexities $$a(\textrm{AND}_2,\pi_{1/2}) = 3/2 \quad \textrm{and} \quad w(\textrm{AND}_2,\pi_{1/2}) = 5/4.$$ 
	For some fixed $p_0 \in (1/2,1]$, we consider the generalized algorithm $A_{p_0}$ which initially asks the $(1-{p_0})/{p_0}$ question for bit 1 (meaning that it picks bit $1$ and goes to the state $(x_1(T_{\{1-{p_0},{p_0}\}}),1/2)$). If $x_1(T_{\{1-{p_0},{p_0}\}})=1-{p_0}$, it asks the $0/1$ question for bit $1$ in the second step, and then the $0/1$ question for bit $2.$
	If $x_1(T_{\{1-{p_0},{p_0}\}})={p_0}$, it asks the $0/1$ question for bit $2$ in the second step, and then the $0/1$ question for bit $1$. One can check that the $\Gamma$-cost of $A_{p_0}$ for determining $\textrm{AND}_2$ is 
	\begin{equation}
		c_\Gamma(A_{p_0},\textrm{AND}_2) = 3/2 + \frac{\Gamma({p_0}) -2{p_0} +1}{4}. 
	\end{equation} 
	We deduce that $a_\Gamma(\textrm{AND}_2,\pi_{1/2}) < a(\textrm{AND}_2,\pi_{1/2})$ if $\Gamma({p_0}) < 2{p_0}-1$. 
\end{example}

We are interested in the following question.
\begin{question} \label{question:algo-gamma}
	Given a Boolean function $f$, for which cost functions $\Gamma$ do we have
	\begin{equation}
		a_\Gamma(f,\pi_{1/2}) =	a(f,\pi_{1/2})\ ?
	\end{equation}
\end{question}

It is interesting to compare our notion of distributional $\Gamma$-algorithmic complexity with the notion of Gross \cite{Gross2022}. On the one hand, his concept of ``fractional query algorithms'' is more general than the ``generalized algorithms'' defined above since it allows for queries, starting at $p_i$, that let $x_i$ evolve up to the hitting time $T_{\{q,p\}}$, where $p_i \in (q,p)$. Unlike in our model, the amount of information revealed by the algorithm is therefore not monotone in time. On the other hand, Gross assumes cost to be quadratic which corresponds to the specific choice of $\Gamma(p)=(2p-1)^2$ in our model. We also refer the reader to \cite{Jacka2011} where, again considering more general algorithms and quadratic cost, the authors study the special case of the MAJORITY function on $n=3$ bits.

\subsection{Distributional $(p,\kappa)$-algorithmic complexity}\label{subsec:dist-p-kappa-alg-com}	

In this subsection, we study the family of generalized algorithms
\begin{equation}
	\mathcal A_p = \{\textrm{algorithms which are only allowed to ask} \ (1-p)/p\ \textrm{or}\ 0/1\ \textrm{questions}\}
\end{equation}
for some fixed $p \in (1/2,1)$. We note that the $\Gamma$-cost of an algorithm $A \in \mathcal A_p$ depends on $\Gamma$ only through $\Gamma(p) \in [0,1]$. Hence, we can represent the cost of a $(1-p)/p$ question by a number $\kappa \in [0,1]$. We write $c_{p,\kappa}(A,f)$ for $c_\Gamma(A,f)$ and naturally define the \emph{distributional $(p,\kappa)$-algorithmic complexity of $f$} as 
\begin{equation}
	a_{p,\kappa}(f,\pi_{1/2}) = \min_{A \in \mathcal A_p} c_{p,\kappa}(A,f).
\end{equation}
An algorithm $A \in \mathcal A_p$ is called \emph{$(p,\kappa)$-optimal for $f$} if it achieves the above minimum.
We ask the analogue of Question \ref{question:algo-gamma} in this simplified setting.
\begin{question} \label{question:algo-p-kappa}
	Given a Boolean function $f$, for which $p$ and $\kappa$ do we have
	\begin{equation}
		a_{p,\kappa}(f,\pi_{1/2}) =	a(f,\pi_{1/2})\ ?
	\end{equation}
\end{question}

It will be convenient to define for $1 \le i \le n$, the random variables
\begin{equation}
	Z_i := \mathbf{1}_{x_i(T_{\{1-p,p\}})=p} \quad \text{and} \quad X_i := x_i(T_{\{0,1\}}).
\end{equation}
We note that the answers to the questions that any algorithm $A \in \mathcal A_p$ asks only depend on these random variables. $Z_i$ and $X_i$ are Ber($1/2$)-distributed, but not independent as $\mathbb P_x[X_i = Z_i] = p$. We now give an alternative description of algorithms in $\mathcal A_p$, which will turn out to be useful.
\begin{definition}\label{defn:algo-as-partition}
	A sequence  of random partitions $A = (A_t)_{0\le t \le 2n} = (A^0_t,A^1_t,A^2_t)_{0 \le t \le 2n}$ of $[n]$ is an \emph{algorithm} if 
	\begin{enumerate}
		\item[(i)] $A_0 = ([n],\emptyset,\emptyset)$ and $A_{2n} = (\emptyset,\emptyset,[n])$,
	\end{enumerate}
	and for every $t \in \{1,\ldots,2n\}$,
	\begin{enumerate}
		\item[(ii)] $A_{t}$ is obtained from $A_{t-1}$ by either moving an element from $A^0_{t-1}$ to $A^1_{t-1}$ or from $A^1_{t-1}$ to $A^2_{t-1}$,
		\item[(iii)] $A_{t}$ is measurable with respect to $A_{t-1}$, $(Z_i)_{i\in A_{t-1}^1}$, and $(X_i)_{i\in A_{t-1}^2}$.
	\end{enumerate} 
\end{definition}

The random partition represents bits for which we have no information ($A^0$), bits for which we have partial information ($A^1$), and bits for which we have full information ($A^2$). Moving bit $i$ from $A^0$ to $A^1$ corresponds to querying $Z_i$ or equivalently asking the $(1-p)/p$ question for bit $i$, thereby providing partial information about $X_i$. Also moving bit $i$ from $A^1$ to $A^2$ corresponds to querying $X_i$ after already having queried $Z_i$, or equivalently asking the $0/1$ question for bit $i$ after already having asked the $(1-p)/p$ question for bit $i$.

An algorithm $A$ is called \emph{induced} if it always queries $X_i$ immediately after $Z_i$; equivalently $A_{t}^1$ has at most one element for every $t$. Clearly, there is a one-to-one correspondence between the subfamily of induced algorithms, denoted by $\mathcal I \subseteq \mathcal A_p$, and the family of standard algorithms for $f$.
Given a Boolean function $f$ on $n$ bits and an algorithm $A \in \mathcal A_p$, we define the (random) first  time at which the algorithm determines $f$ as
\begin{equation}
	T_A = T_{A}(f) := \min\left\{t \ge 0 : f \ \text{is determined by} \ (X_i)_{i \in A_t^2}\right\}.
\end{equation}
We note that the cost of $A \in \mathcal A_p$ to determine $f$ is given by
\begin{equation}
	c_{p,\kappa}(A,f) =  \E_x \left[ \kappa \cdot \abs{A_{T_A}^1} + \abs{A_{T_A}^2 } \right].
\end{equation}

\begin{remark} \label{rem:kappa-1}
	For $\kappa =1$ and any $p \in (1/2,1)$, we have $a_{p,1}(f,\pi_{1/2}) = a(f,\pi_{1/2})$ since asking a $0/1$ question after the corresponding $(1-p)/p$ question is of no disadvantage as it has no additional cost. In other words, 
	\begin{equation}
		\min_{A \in \mathcal A_{p}} \E_x \left[\abs{A_{T_A}^1} + \abs{A_{T_A}^2 } \right] =  \min_{I \in \mathcal I} \E_x\left[ \abs{I_{T_I}^2}\right]. 
	\end{equation}
\end{remark}
The next easy proposition establishes monotonicity in $p$ and $\kappa$ with respect to Question \ref{question:algo-p-kappa}.
\begin{proposition}\label{prop:simplified-model-monotonicity}
	Consider $(p,\kappa),\ (p',\kappa') \in (1/2,1) \times [0,1]$. If $p \ge p'$ and $\kappa' \ge \kappa$, then we have 
	\begin{equation}
		a_{p,\kappa}(f,\pi_{1/2}) = a(f,\pi_{1/2}) \implies a_{p',\kappa'}(f,\pi_{1/2}) = a(f,\pi_{1/2})
	\end{equation}
\end{proposition}
\begin{proof}
	First, fix $p$, let $\kappa' \ge \kappa$ and assume that there exists an induced $(p,\kappa)$-optimal algorithm $I \in \mathcal I$ for $f$. We note that the cost of any algorithm $A \in \mathcal A_p$ to determine $f$ is non-decreasing in $\kappa$, and so also the distributional $(p,\kappa)$-algorithmic complexity of $f$. Moreover, the cost of any induced algorithm to determine $f$ is constant in $\kappa$. Therefore,
	\begin{equation}
		a_{p,\kappa}(f,\pi_{1/2}) \le a_{p,\kappa'}(f,\pi_{1/2}) \le c_{p,\kappa'}(I,f) = c_{p,\kappa}(I,f),
	\end{equation}
	and we deduce from the $(p,\kappa)$-optimality of $I$ that both inequalities are in fact equalities. Hence, $a(f,\pi_{1/2}) = 	a_{p,\kappa}(f,\pi_{1/2})$ implies $a(f,\pi_{1/2}) = a_{p,\kappa'}(f,\pi_{1/2})$.
	
	Second, fix $\kappa$ and let $p \ge p'$. Since $a_{p',\kappa}(f,\pi_{1/2}) \le a(f,\pi_{1/2})$, it suffices to show that $a_{p,\kappa}(f,\pi_{1/2}) \le a_{p',\kappa}(f,\pi_{1/2})$. To this end, we consider the family $\mathcal A_{p',p}$ of generalized algorithms, which are only allowed to ask $(1-p')/p'$, $(1-p)/p$ or $0/1$ questions, and we associate cost $\kappa$ to asking a $(1-p')/p'$ question, cost $0$ to asking a $(1-p)/p$ question (after the corresponding $(1-p')/p'$ question), and cost $1-\kappa$ to asking a $0/1$ question (after the corresponding $(1-p)/p$ question). Since it is clearly optimal for any algorithm $A \in \mathcal A_{p',p}$ to always ask the $(1-p)/p$ question directly after the $(1-p')/p'$ question, we have using obvious notation
	\begin{equation}
		a_{p',\kappa}(f,\pi_{1/2}) = \min_{A \in \mathcal A_{p'}} c_{p',\kappa}(A,f) \ge \min_{A \in \mathcal A_{p',p}} c_{p',p,\kappa}(A,f) = \min_{A \in \mathcal A_{p}} c_{p,\kappa}(A,f) = a_{p,\kappa}(f,\pi_{1/2}).
	\end{equation}
\end{proof}

For $p \in (1/2,1)$ and $\kappa \in [0,1]$, we define 
\begin{equation}
	\kappa_c(f,p) = \inf\left\{\kappa : a_{p,\kappa}(f,\pi_{1/2}) = a(f,\pi_{1/2})\right\},
\end{equation}
and 
\begin{equation}
	p_c(f,\kappa) = \sup\left\{p : a_{p,\kappa}(f,\pi_{1/2}) = a(f,\pi_{1/2})\right\}.
\end{equation}
By the previous proposition $p \mapsto \kappa_c(f,p)$ and $\kappa \mapsto p_c(f,\kappa)$ are (weakly) increasing functions. For the AND function on $n=2$ bits, it can easily be shown based on the computations in Example \ref{ex:gamma-algo-AND-2} that $\kappa_c(\textrm{AND}_2,p) = 2p-1$.

\begin{theorem} \label{thm:composition-simplified-model}
	For every Boolean function $f$ on $n \ge 1$ bits and $p \in (1/2,1)$,
	\begin{equation}\label{eq:thm:composition-simplified-model}
		\kappa_c(f,p) \le  \kappa_0(n,p) :=\frac{1}{1 + \frac{1}{n} \left(\frac{1-p}{2}\right)^n}\ .
	\end{equation}
\end{theorem}

We divide the proof into two lemmas. 

\begin{lemma}\label{lem:simplified-framework-1}
	Let $f$ be a Boolean function on $n\ge 1$ bits. If $\kappa > \kappa_0(n,p)$, then every  algorithm $A \in \mathcal A_p$ that is $(p,\kappa)$-optimal for $f$ satisfies 
	\begin{equation} \label{eq:lem:simplified-framework-1}
		\E_x \left[\abs{A_{T_A(f)}^1} + \abs{A_{T_A(f)}^2 } \right] 
		= \min_{I \in \mathcal I} \E_x\left[ \abs{I_{T_I(f)}^2}\right].
	\end{equation}%
	Moreover, if \eqref{eq:lem:simplified-framework-1} holds and if $Z_i$ is a $(p,\kappa)$-optimal first query among all algorithms in $\mathcal A_p$, then $Z_i$ is also an optimal first query among all algorithms in $\mathcal I$. 
\end{lemma}
By Remark \ref{rem:kappa-1}, the inequality $\ge$ in \eqref{eq:lem:simplified-framework-1} holds for all $\kappa$. AND$_2$ provides an easy counterexample showing that equality can fail for small $\kappa$. 
The following example shows that $\kappa$ large is also a necessary assumption to ensure that an optimal first query among all algorithms in $\mathcal A_p$ is also an optimal first query among all algorithms in $\mathcal I$. 
\begin{example} \label{ex:address-7}
	Consider a variant of the ADDRESS function on $n=7$ bits, defined by 
	\begin{equation}
		f(x_1,\ldots,x_7) = \begin{cases} 	x_1 \oplus x_2 &\textrm{if}\ x_5 \oplus x_6 \oplus x_7 = 1, \\
			x_3 \oplus x_4 &\textrm{if}\ x_5 \oplus x_6 \oplus x_7 = 0. \\
		\end{cases}
	\end{equation}
	On the one hand, the algorithm $I$ which first queries $X_5$, $X_6$, $X_7$, and then either $X_1$, $X_2$ or $X_3$, $X_4$ depending on the value of $X_5 \oplus X_6 \oplus X_7$, has cost 5. In fact, every induced algorithm that is optimal among $\mathcal I$ starts with querying $X_5$, $X_6$ and $X_7$ in some order. On the other hand, consider the algorithm $A$ which first queries $Z_1$, $Z_2$, $Z_3$, $Z_4$. If $Z_1 \oplus Z_2 = Z_3 \oplus Z_4$, it then queries $X_1$, $X_2$, $X_3$, $X_4$ and if needed also $X_5$, $X_6$, $X_7$. Otherwise, it queries $X_5$, $X_6$, $X_7$ , and then either $X_1$, $X_2$ or $X_3$, $X_4$ depending on the value of $X_5 \oplus X_6 \oplus X_7$. Note that $A$ has expected cost $9/2 + \kappa + 6p(1-p)(p^2 + (1-p)^2)$, which is strictly smaller than $5$ if $\kappa < 1/2 - 6p(1-p)(p^2+(1-p)^2)$. Furthermore, for such $\kappa$, one can check that $Z_1$, $Z_2$, $Z_3$, $Z_4$ are the only optimal first queries among $\mathcal A_p$.
\end{example}

To prove the lemma, it will be convenient to define for $1 \le i \le n$,
\begin{equation}
	Y_i := \mathbf{1}_{Z_i = X_i}.
\end{equation}
Note that $Y_i = 0$ means the Brownian motion $x_i$ terminated at $0$ (respectively $1$) even though it hit $p \in (1/2,1)$ before $1-p$ (respectively $1-p$ before $p$). 

\begin{proof}[Proof of Lemma \ref{lem:simplified-framework-1}]
	Fix a Boolean function $f$ on $n \ge 1$ bits. By Remark \ref{rem:kappa-1}, 
	\begin{equation}\label{eq:7}
		\min_{B \in \mathcal A_p} \E_x \left[\abs{B_{T_B}^1} + \abs{B_{T_B}^2 } \right] = \min_{J \in \mathcal I} \E_x\left[ \abs{J_{T_J}^2}\right], 
	\end{equation}%
	and in the same way, we have for every $y_1,\ldots,y_n \in \{0,1\}$, 
	\begin{equation}\label{eq:6}
		\begin{split}
			\min_{B \in \mathcal A_p} \E_x \left[\abs{B_{T_B}^1} + \abs{B_{T_B}^2 } \mid Y_1 = y_1,\ldots,Y_n = y_n  \right] \\ = \min_{J \in \mathcal I} \E_x\left[ \abs{J_{T_J}^2}\mid Y_1 = y_1,\ldots,Y_n = y_n  \right]. \\ \end{split}
	\end{equation}%
	We now fix a $(p,\kappa)$-optimal algorithm $A \in \mathcal A_p$ and an induced algorithm $I \in \mathcal I$ achieving the minimum in the right-hand side of  \eqref{eq:7}. Since the cost of an induced algorithm is independent of $(Y_i)_{i=1}^n$, $I$ also achieves the minimum in  the right-hand side of \eqref{eq:6} for all $y_1,\ldots,y_n \in \{0,1\}$. Thus, for every $y_1,\ldots,y_n \in \{0,1\}$,	
	\begin{align}\label{eq:5}
		\E_x \left[\abs{A_{T_A}^1} + \abs{A_{T_A}^2 } \mid Y_1 = y_1,\ldots,Y_n = y_n  \right] &\ge \E_x\left[ \abs{I_{T_I}^2}\mid Y_1 = y_1,\ldots,Y_n = y_n\right] \\ &= \E_x\left[ \abs{I_{T_I}^2}\right].
	\end{align}
	Towards a contradiction, assume that $\kappa > \kappa_0(n,p)$ and that in \eqref{eq:lem:simplified-framework-1}, the left-hand side is strictly larger than the right-hand side. Then inequality \eqref{eq:5} must be strict for some $\hat{y}_1,\ldots,\hat{y}_n \in \{0,1\}$ and so we must have
	\begin{equation}\label{eq:0}
		\E_x\left[\abs{A^1_{T_A}} + \abs{A^2_{T_A}} \mid Y_1 = \hat{y}_1, \ldots, Y_n = \hat{y}_n\right] \ge 2^{-n} + \E_x\left[\abs{I^2_{T_I}} \mid Y_1 = \hat{y}_1, \ldots, Y_n = \hat{y}_n\right]
	\end{equation}
	since the conditional expectations are multiples of $2^{-n}$. Indeed, they only depend on the randomness of $Z_1,\ldots,Z_n$, which are independent of $Y_1,\ldots,Y_n$. Comparing the $(p,\kappa)$-costs of $A$ and $I$, we obtain
	\begin{align*}
		&c_{p,\kappa}\left(A,f\right) - c_{p,\kappa}\left(I,f\right) \\
		&= (1-\kappa) \cdot \underbrace{\left(\E_x\left[\abs{A^2_{T_A}}\right] -  \E_x\left[\abs{I^2_{T_I}}\right]\right)}_{\ge -n} + \kappa \cdot 	\underbrace{\left(\E_x\left[\abs{A^1_{T_A}} + \abs{A^2_{T_A}}\right] - \E_x\left[\abs{I^2_{T_I}}\right]\right)}_{\ge 2^{-n} (1-p)^{n}} \\
		&\ge - n + \kappa \cdot \left(n + 2^{-n} (1-p)^{n}\right) > 0,
	\end{align*}
	where we bounded the difference $\E_x[\abs{A^1_{T_A}} + \abs{A^2_{T_A}}] - \E_x[\abs{I^2_{T_I}}]$ from below using \eqref{eq:5} and \eqref{eq:0} and $\P_x\left[ Y_1 = \hat{y}_1, \ldots, Y_n = \hat{y}_n\right] \ge \min\{1-p,p\}^{n} = (1-p)^n$. This contradicts the $(p,\kappa)$-optimality of $A$. This concludes the proof of \eqref{eq:lem:simplified-framework-1}. The final statement is straightforward and left to the reader.
\end{proof} 

\begin{lemma}\label{lem:simplified-framework-2}
	Let $f$ be a Boolean function on $n \ge 1$ bits. For every algorithm $A \in \mathcal A_p$,
	\begin{equation} \label{eq:lem:simplified-framework-2}
		\E_x \left[\abs{A_{T_A(f)}^1} + \abs{A_{T_A(f)}^2 } \right] = \min_{I \in \mathcal I} \E_x\left[ \abs{I_{T_I(f)}^2}\right] \quad \implies \quad \E_x \left[\abs{A_{T_A(f)}^1}\right] = 0.
	\end{equation}%
\end{lemma}
\begin{proof}
	Recall that for $p \in (1/2,1)$, we have defined the random variables 
	\begin{equation}
		Z_i = \mathbf{1}_{x_i(T_{\{1-p,p\}})=p}, \quad X_i = x_i(T_{\{0,1\}}), \quad \text{and} \quad Y_i = \mathbf{1}_{Z_i = X_i}
	\end{equation}%
	for $1 \le i \le n$. We naturally extend these definitions to $p=1/2$ by choosing $(Z_i)_{1 \le i \le n}$ to be Ber($1/2$)-distributed random variables that are independent of each other and of the Brownian motions $(x_i)_{1 \le i \le n}$ in this case. For every $p \in [1/2,1)$, we note that $Y_i \sim \textrm{Ber}(p)$, $Y_i$ is independent of $Z_i$, and $X_i = Z_i \cdot Y_i + (1-Z_i)\cdot (1-Y_i)$. 
	
	Fix a Boolean function $f$ on $n \ge 1$ bits, an algorithm $A \in \mathcal A_p$ and an algorithm $I \in \mathcal I$ that is optimal among all induced algorithms.
	As argued at the beginning of the proof of Lemma \ref{lem:simplified-framework-1} (see \eqref{eq:6}), we have 	
	\begin{equation}\label{eq:2} 
		\E_x \left[\abs{A_{T_A}^1}+\abs{A_{T_A}^2 } \mid Y_1 = y_1,\ldots,Y_n = y_n  \right]  \ge \E_x\left[ \abs{I_{T_I}^2}\mid Y_1 = y_1,\ldots,Y_n = y_n\right]
	\end{equation}
	for every $y_1,\ldots,y_n \in \{0,1\}$.	From now on, we assume $\E_x [\abs{A_{T_A}^1} + \abs{A_{T_A}^2 } ] = \min_{I \in \mathcal I} \E_x[ \abs{I_{T_I}^2}]$. Equation \eqref{eq:2} then implies 
	\begin{equation}\label{eq:8} 
		\E_x \left[\abs{A_{T_A}^1}+\abs{A_{T_A}^2 } \mid Y_1 = y_1,\ldots,Y_n = y_n  \right]  = \E_x\left[ \abs{I_{T_I}^2}\mid Y_1 = y_1,\ldots,Y_n = y_n\right]
	\end{equation}
	for every $y_1,\ldots,y_n \in \{0,1\}$.
	Viewed as a sequence of random partitions (see Definition \ref{defn:algo-as-partition}), it makes sense to also consider the algorithm $A$ for other values of $p \in [1/2,1)$.
	Neither of the two sides in \eqref{eq:8} depends on $p$ and we thus deduce that for \emph{every} $p\in [1/2,1)$, we have 
	\begin{equation}
		\E_x \left[\abs{A_{T_A}^1}+\abs{A_{T_A}^2 }  \right]  = \E_x\left[ \abs{I_{T_I}^2}\right].
	\end{equation}
	However, for $p=1/2$, we actually have no information about the values  of the bits in $A^1$, meaning that given the $Z_i$'s with $i \in A^1$, the $X_i$'s with $i \in A^1$ are still independent and  Ber(1/2)-distributed. Hence, the algorithm $A$ can be viewed as a randomized algorithm where the external randomness is coming from the $Z_i$ bits. It now follows from the optimality of $I$ among all induced algorithms that we must have 
	\begin{equation}
		\E_x \left[\abs{A_{T_A}^2 }\right]  \ge \E_x\left[ \abs{I_{T_I}^2}\right], \quad \text{and thus,} \quad 	\E_x \left[\abs{A_{T_A}^1 }\right] = 0.
	\end{equation}
	We deduce that for every $y_1,\ldots,y_n \in \{0,1\}$ and $p=1/2$,
	\begin{equation}
		\E_x \left[\abs{A_{T_A}^1 } \mid Y_1 = y_1,\ldots,Y_n = y_n\right] = 0.
	\end{equation}
	This implies  $\E_x [\abs{A_{T_A}^1 }] = 0$ for every $p \in [1/2,1)$, thereby concluding the proof.
\end{proof}

We are now ready to prove the theorem.

\begin{proof}[Proof of Theorem \ref{thm:composition-simplified-model}]
	We need to show that for any Boolean function $f$ on $n \ge 1$ bits and $\kappa > \kappa_0(n,p)$, we have  
	\begin{equation} \label{eq:9}
		a_{p,\kappa}(f,\pi_{1/2}) =  a(f,\pi_{1/2}).
	\end{equation}%
	Now, fix such an $f$ and $\kappa$. Combining Lemmas \ref{lem:simplified-framework-1} and \ref{lem:simplified-framework-2}, we deduce that every $(p,\kappa)$-optimal algorithm $A \in \mathcal A_p$ for $f$ satisfies 
	\begin{equation} 
		\E_x \left[\abs{A_{T_A(f)}^1}\right] = 0.
	\end{equation}%
	Hence, for every $(p,\kappa)$-optimal algorithm $A \in \mathcal A_p$ for $f$, we have
	\begin{equation} 
		c_{p,\kappa}(A,f) = \E_x \left[\abs{A_{T_A(f)}^2}\right] = c_{p,1}(A,f),
	\end{equation}%
	and so
	\begin{equation}
		a_{p,\kappa}(f,\pi_{1/2}) = a_{p,1}(f,\pi_{1/2}) = a(f,\pi_{1/2}),
	\end{equation}%
	where the last equality is due to  Remark \ref{rem:kappa-1}.
	This concludes the proof.
\end{proof}

\begin{remark}
	Equality \eqref{eq:9} implies that for $\kappa>\kappa_0(n,p)$, any optimal first query among $\mathcal I$ is also an optimal first query among $\mathcal A_p$. However,  Example \ref{ex:address-7} shows that this might not be the case for $\kappa$ small.
\end{remark}

Recall that we have seen that $\kappa_c(\textrm{AND}_2,p) = 2p -1$. Moreover, one can check that for any permutation invariant Boolean function $f$ on $n=3$ bits, $\kappa_c(f,p) \in \{0,2p -1\}$. This naturally leads to the following questions. 
\begin{question}\label{question:kappa-critical-symmetric} Let $\mathcal S$ denote the set of permutation invariant Boolean functions. Do we have 
	\begin{equation}
		\sup_{f \in \mathcal{S}}\ \kappa_c(f,p) = 2p -1 \ ?
	\end{equation}
\end{question}
\begin{question}\label{question:kappa-critical} Do we have 
	\begin{equation}
		\sup_{f}\ \kappa_c(f,p) = 2p-1  \quad \textrm{or} \quad \sup_{f}\ \kappa_c(f,p) < \infty   \ ?
	\end{equation}
\end{question}

\section{Distributional algorithmic complexity under composition and an outlook on asymptotic separation}\label{sec:alg-comp-composition}
	
	While this section does not contain theorems, it provides a natural conjecture together with a connection to asymptotic separation of distributional algorithmic and subcube partition complexity. It also brings us back to the discussion at the end of Sections \ref{sec:asymp-separation}.
	
	Recall that for two Boolean functions $f$, $g$ on $n$ resp., $m$ bits, the \emph{composition} $f \circ g$ is the Boolean function on $n\cdot m$ bits  defined by
	\begin{equation}
		f\circ g (x^1_{1},\ldots,x^1_{m},\ldots,x^n_{1},\ldots,x^n_{m}) = f\left(g(x^1_{1},\ldots,x^1_{m}),\ldots,g(x^n_{1},\ldots,x^n_{m})\right).
	\end{equation} 
	Moreover, recall that the PARITY function on $n$ bits is defined by 
	\begin{equation}
		\textrm{PAR}_n (x_1,\ldots,x_n) = \begin{cases} 	1 &\textrm{if} \sum_{i=1}^n x_i \ \textrm{is odd}, \\
			0 &\textrm{if} \sum_{i=1}^n x_i \ \textrm{is even}.
		\end{cases}
	\end{equation}
	We believe that the ideas of the proof of Theorem \ref{thm:composition-simplified-model} could lead to the following result which we therefore state as a conjecture.
	\begin{conjecture}\label{thm:composition}
		For every $n,m \ge 1$, there exists $k_0(n,m) \in \mathbb N$, such that for any Boolean functions $f$ and $g$ on $n \ge 1$ respectively $m \ge 1$  bits and for any $k \ge k_0(n,m)$, we have 
		$$a\left(f\circ (g \oplus \textrm{PAR}_k),\pi_{1/2}\right) = a(f,\pi_{1/2}) \cdot a(g \oplus \textrm{PAR}_k,\pi_{1/2}). $$
	\end{conjecture}
	
	The above conjecture would then imply that for every $n \ge 1$ and for every $L \ge 1$, there exists $k_1(n,L) \in \mathbb N$ such that for every Boolean function $f$ on $n \ge 1$ bits and for every $k \ge k_1(n,L)$, we have 
		\begin{equation}
			a\left((f\oplus \textrm{PAR}_k)^\ell,\pi_{1/2}\right) = \left(a(f\oplus \textrm{PAR}_k,\pi_{1/2})\right)^\ell, \quad \forall 1 \le \ell \le L.
		\end{equation}
	This naturally leads to the next question.
	\begin{question}
		For which Boolean functions $f$ does there exist $k_1(f) \in \mathbb N$ such that for every $k \ge k_1(f)$, we have 
		\begin{equation}
			a\left((f\oplus \textrm{PAR}_k)^\ell,\pi_{1/2}\right) = \left(a(f\oplus \textrm{PAR}_k,\pi_{1/2})\right)^\ell, \quad \forall 1 \le \ell <\infty ?
		\end{equation}
		In particular, does this hold for $f = \textrm{A-EQ}_3$ (see Example \ref{ex:constant-function})?
	\end{question}
	
	While the parity function trivially provides such an example, one approach in achieving asymptotic separation of algorithmic and subcube partition complexity is to find an example of a Boolean function $f$ for which  $sc(f,\pi_p) < a(f,\pi_p)$ and which satisfies the above. 
	As explained at the end of Section \ref{sec:asymp-separation}, our hope would be that the second question has a positive answer, implying that the sequence 
	\begin{equation}
		\left((\textrm{A-EQ}_3 \oplus \textrm{PAR}_k)^\ell\right)_{\ell \ge 1}
	\end{equation}
	asymptotically separates distributional algorithmic and subcube parition complexity for some sufficiently large $k\ge 1$.

\subsection*{Acknowledgments}

In 2004, Oded Schramm and the second author began to look at some of the questions addressed in this paper and obtained some preliminary results and examples. The project was put on ice in 2005. When the second author visited Zurich in 2019, it turned out that the first author was looking at similar questions during a semester project in 2018 supervised by Vincent Tassion and also had some preliminary results. It then became natural to combine forces. We hereby acknowledge Oded for his contributions to this paper.

The first author is grateful to Vincent Tassion for inspiring questions and discussions, and for supporting his visits to Gothenburg. Both authors would like to thank Vincent Tassion and Aran Raoufi for stimulating discussions during the second author's visits to Zurich, as well as the Forschungsinstitut für Mathematik (FIM) at ETH Zurich for its hospitality during these visits.

The first author is part of NCCR SwissMAP and has received funding from the European Research Council (ERC) under the European Union’s Horizon 2020 research and innovation program (grant 851565). The second author acknowledges the support of the Swedish Research Council (grant 2020-03763).

\small
\bibliographystyle{alpha}
\bibliography{refs}

\end{document}